\documentclass{article}
\usepackage[a4paper,margin=1.5in]{geometry} 
\usepackage{graphicx}
\usepackage{amsmath,amssymb,amsthm}
\usepackage{stmaryrd}
\usepackage[colorlinks]{hyperref}
\usepackage{cleveref}
\usepackage{imakeidx}
\usepackage{faktor}
\usepackage{xcolor}
\usepackage{soul}
\usepackage{caption}
\usepackage{shuffle}
\usepackage{authblk}
\makeindex[intoc,name=general]
\usepackage{tikz}

\usepackage{pgfplots}
\usepackage{subcaption} 
\usepackage{mathtools}
\usepackage{bbm}
\usepackage[style=numeric, sortcites]{biblatex}
\hypersetup{
    colorlinks=true,         
    citecolor=myGreen,           
     linkcolor=darkred
}

\usepackage[colorinlistoftodos, bordercolor=orange, backgroundcolor=orange!20, linecolor=orange, textsize=scriptsize]{todonotes}

\newcommand{\mb}{\mathbb}
\newcommand{\mc}{\mathcal}

\newcommand{\tensor}{\otimes}

\newcommand{\Axis}{\mathsf{Axis}}

\DeclareMathOperator{\GL}{GL}

\newcommand\DEF[1]{\emph{#1}}

\newcommand\leftisomto{\stackrel{\sim}{\smash{\longleftarrow}\rule{0pt}{0.4ex}}}
\newcommand{\inlmatr}[1]{\begin{psmallmatrix} #1 \end{psmallmatrix}}
\newcommand{\matr}[1]{\begin{pmatrix} #1 \end{pmatrix}}

\definecolor{darkgreen}{RGB}{0,150,0}

\definecolor{myGreen}{rgb}{0.18039216 0.49803922 0.09411765}
\definecolor{oldlace}{rgb}{0.99, 0.96, 0.9}
\definecolor{moccasin}{rgb}{0.93333333 0.49019608 0.04313725} 
\definecolor{palegreen}{rgb}{0.6, 0.98, 0.6}
\definecolor{green(html/cssgreen)}{rgb}{0.0, 0.5, 0.0}
\definecolor{harlequin}{rgb}{0.25, 1.0, 0.0}
\definecolor{lasallegreen}{rgb}{0.03, 0.47, 0.19}
\definecolor{kellygreen}{rgb}{0.3, 0.73, 0.09}
\definecolor{forestgreen(web)}{rgb}{0.2, 0.8, 0.2}
\definecolor{darktangerine}{rgb}{1.0, 0.66, 0.07}
\definecolor{fashionfuchsia}{rgb}{0.96, 0.0, 0.63}
\definecolor{darkred}{rgb}{0.65, 0.0, 0.0}
\definecolor{darkRed}{rgb}{0.65, 0.0, 0.0}
\definecolor{coquelicot}{rgb}{0.9, 0.22, 0.0}
\definecolor{blue-violet}{rgb}{0.54, 0.17, 0.89}
\definecolor{airforceblue}{rgb}{0.36, 0.54, 0.66}
\definecolor{tealBlue}{rgb}{0.21, 0.46, 0.53}
\definecolor{indigoWeb}{rgb}{0.29, 0.0, 0.51}
\definecolor{CeruleanBlue}{rgb}{0.16, 0.32, 0.75}
\definecolor{darkBlue}{rgb}{0.16, 0.32, 0.75}
\definecolor{darkGrey}{rgb}{0.66, 0.66, 0.66}
\definecolor{lightGrey}{rgb}{0.33, 0.33, 0.33}
\definecolor{ballblue}{rgb}{0.392,0.513,0.725}
\definecolor{brown}{rgb}{0.45882353 0.27058824 0.00784314}
\definecolor{babypink}{rgb}{0.878, 0.345, 0.12549}
\definecolor{darkorange}{rgb}{0.894, 0.866, 0.125}
\definecolor{aquamarine}{rgb}{0.5, 1.0, 0.83}
\definecolor{blue_munsell}{rgb}{0.0, 0.5, 0.69}
\definecolor{brass}{rgb}{0.71, 0.65, 0.26}
\definecolor{brightlavender}{rgb}{0.75, 0.58, 0.89}

\DeclareMathOperator{\Hom}{Hom}

\theoremstyle{theorem}
\newtheorem{lemma}{Lemma}[section]
\newtheorem{theorem}[lemma]{Theorem}
\newtheorem*{theorem*}{Theorem}
\newtheorem{proposition}[lemma]{Proposition}
\newtheorem{corollary}[lemma]{Corollary}
\newtheorem*{corollary*}{Corollary}
\newtheorem{conjecture}[lemma]{Conjecture}
\newtheorem{question}[lemma]{Question}

\theoremstyle{definition}
\newtheorem{definition}[lemma]{Definition}
\newtheorem{example}[lemma]{Example}
\newtheorem{caveat}[lemma]{Caveat}
\newtheorem{remark}[lemma]{Remark}
\newtheorem{rmknot}[lemma]{Remark and Notation}

\theoremstyle{remark}

\title{Signature matrices of membranes}
\author[1]{Felix Lotter}
\author[2]{Leonard Schmitz}
\affil[1]{MPI MiS Leipzig, Germany}
\affil[2]{TU Berlin, Germany}
\date{}

\begin{document}
\maketitle

\begin{abstract}
The signature of a membrane is a sequence of tensors whose entries are iterated integrals. We study algebraic properties of membrane signatures, with a focus on signature matrices of polynomial and piecewise bilinear membranes. Generalizing known results for path signatures, we show that the two families of membranes admit the same set of signature matrices and we examine the corresponding affine varieties. In particular, we prove that there are no algebraic relations on signature matrices of membranes, in contrast to the path case. We complement our results by a linear time algorithm for the computation of signature tensors for piecewise bilinear membranes.
\end{abstract}
\textbf{Keywords:} two-parameter signatures,  affine algebraic varieties, matrix congruence, geometry of paths and membranes, iterated integrals
\\\\
\textbf{MSC codes:}
60L10, 14Q15, 13P25

\section{Introduction}

A \textit{path} is a continuous map $X:[0,1] \to \mb R^d$ with piecewise continuously differentiable coordinate functions. The signature of a path was introduced by Chen in his seminal work \cite{Chen1954IteratedIA} and has since become a central construction in the theory of rough paths. The \emph{second level signature} $S(X)\in\mathbb{R}^{d\times d}$ of a path $X$ is a matrix whose entries are iterated integrals
\begin{align*}
S(X)_{i,j}&:=
\int_0^1\int_0^{t_2}\dot X_{i}(t_1)\,\dot X_{j}(t_2)\;\mathrm{d}t_1\mathrm{d}t_2
\end{align*}
corresponding to $1\leq i,j\leq d$, and where $\dot X$ denotes the derivative of $X$. 
These \emph{signature matrices} were studied in \cite[Section 3]{bib:AFS2019} and satisfy several interesting properties:  their symmetric parts always have rank $1$, due to the \emph{shuffle relations} satisfied by path signatures. 
The sets of signature matrices of \textit{polynomial} and \textit{piecewise linear paths} coincide. This set is semi-algebraic and its Zariski closure can be described as a certain determinantal variety. 

\subsection*{The signature of a membrane}
Recently, the signature has been extended to membranes. 
A $d$-dimensional \textit{membrane} is a continuous map $X:[0,1]^2\rightarrow{\mathbb{R}}^d$ where the coordinate functions $X_1, X_2, \dots , X_d $ are (piecewise) continuously differentiable functions in the two arguments. In particular the differentials $\partial_{12}X_j(s,t):=\frac{\partial^2}{\partial s\partial t}X_j(s,t)$ obey the usual rules of calculus. Generalizing the path setting, one can associate a collection of tensors to $X$ via iterated integrals (see \Cref{def:signature}). The associated object was termed \textit{$\mathsf{id}$-signature} in
\cite[Definition 4.1]{diehl2024signature}, \cite[Section 4.7]{GLNO2022}. We call the second level of the \textit{$\mathsf{id}$-signature} the \textit{signature matrix} of a membrane. It is defined as the $\mathbb{R}$-valued $d\times d$ matrix $S(X)$ with entries
\begin{align}\label{eq:membrane_sig_tensor}
S(X)_{i,j}&:=
\int_0^1\int_0^1\int_0^{t_2}\int_0^{s_2}\partial_{12}X_{i}(s_1,t_1)\,\partial_{12}X_{j}(s_2,t_2)\;\mathrm{d}s_1\mathrm{d}t_1\mathrm{d}s_2\mathrm{d}t_2
\end{align}
for every $1\leq i,j\leq d$.

  Although it is easy to see that the $\mathsf{id}$-signature does not satisfy the same algebraic relations as the path signature \cite[Section 4.3]{diehl2024signature}, little is known about its algebraic relations in general. In this paper, inspired by \cite{bib:AFS2019}, we address this question for the case of signature matrices. For this we consider the family of membranes whose component functions are polynomials of order $(m,n)$. Their signature matrices \eqref{eq:membrane_sig_tensor}  form a \textit{semialgebraic subset} of $\mb R^{d\times d}$. To study the polynomials that vanish on these sets, we move to the algebraically closed field $\mb C$ and consider the associated  \emph{Zariski closure}, in the following denoted by $\mathcal{M}_{d,m,n}$. 

\begin{figure}
    \centering
\begin{subfigure}{0.45\textwidth}
        \centering
        \includegraphics{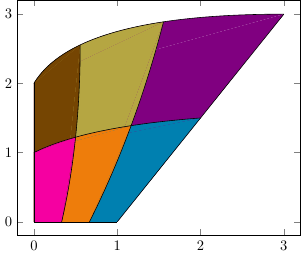}
\end{subfigure}
\hfill
\begin{subfigure}{0.45\textwidth}
        \centering
    \includegraphics{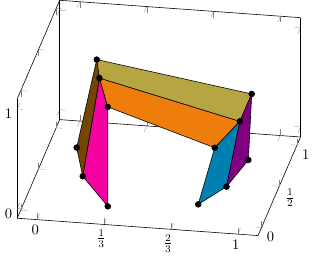}
\end{subfigure}

    \caption{A polynomial membrane in $\mb R^2$ and a piecewise bilinear membrane in $\mb R^3$. The colors correspond to images of rectangles in a uniform partitioning of $[0,1]^2$. See \Cref{ex:closed_formk2} and \Cref{fig:illustrativeMembranes} for a more detailed description.
}
\end{figure}

\subsection*{An implicitization problem}

Our strategy is to view the variety $\mathcal{M}_{d,m,n}$ as the Zariski closure of a certain matrix congruence orbit. For this we use a \textit{dictionary membrane} inspired by \cite{bib:PSS2019}, i.e., a special membrane that yields all polynomial membranes via a suitable linear transform, up to terms that vanish under the signature map.  The  \textit{moment membrane},
    \begin{align}\label{eq:into_momentMembrane}
        \mathsf{Mom}^{m,n}: [0,1]^2 &\to \mathbb R^{mn},(s,t)\mapsto(s,s^2,\dots,s^m)\otimes(t,t^2,\dots,t^n), 
    \end{align}
is a dictionary for polynomial membranes of order $(m,n)$, where
 $\otimes$ denotes the \emph{Kronecker product}. We call this membrane a product membrane since it factorizes into paths. We will see in \Cref{sec:prod_membranes} that the signature matrix of a product membrane is a Kronecker product of the signature matrices of its path factors. If for example $m=n=2$,  then $\mathsf{Mom}^{2,2}(s,t)=(st,st^2,s^2t,s^2t^2)$ has the signature matrix  
\begin{equation}\label{eq:ex_C2}S(\mathsf{Mom}^{2,2})=
\begin{psmallmatrix}
    \frac14&\frac13&\frac13&\frac49\\[4pt]
    \frac16&\frac14&\frac29&\frac13\\[4pt]
    \frac16&\frac29&\frac14&\frac13\\[4pt]
    \frac19&\frac16&\frac16&\frac14
\end{psmallmatrix}
=
\begin{psmallmatrix}
    \frac12&\frac23\\[4pt]
    \frac13&\frac12
\end{psmallmatrix}\otimes\begin{psmallmatrix}
    \frac12&\frac23\\[4pt]
    \frac13&\frac12
\end{psmallmatrix}.
\end{equation}

Since the moment membrane is a dictionary, we can express every polynomial membrane $X$ without linear terms of order $(m,n)$ as a linear transform $X=A\,\mathsf{Mom}^{m,n}$, where the matrix $A$ is chosen accordingly. Note that the membrane signature is invariant under addition of linear terms.
With \emph{equivariance} of the signature (\Cref{lmm:coord-free equivariance}) we obtain the closed formula, 
\begin{equation}\label{eq:intro_matrix_congruence}
    S(X)=A S(\mathsf{Mom}^{m,n})A^\top,
\end{equation} 
through a matrix congruence transform. 
Continuing with our example  
\eqref{eq:ex_C2} and $d=2$, we have 
\begin{equation}\label{eq:polynomialMembran_intro}X=A\,\mathsf{Mom}^{2,2}=\begin{pmatrix}
    a_{1,1}st+a_{1,2}st^2+a_{1,3}s^2t+a_{1,4}s^2t^2\\
    a_{2,1}st+a_{2,2}st^2+a_{2,3}s^2t+a_{2,4}s^2t^2
\end{pmatrix}
\end{equation}
and thus with \eqref{eq:ex_C2} and \eqref{eq:intro_matrix_congruence} the homogeneous polynomial 
\begin{align*}
S(X)_{1,1}=\frac14a_{1, 1}^2 &+ \frac12a_{1, 1}a_{1, 2} + \frac12a_{1, 1}a_{1, 3} + \frac59a_{1, 1}a_{1, 4} + \frac14a_{1, 2}^2\\ &+ \frac49a_{1, 2}a_{1, 3}+ \frac12a_{1, 2}a_{1, 4} + \frac14a_{1, 3}^2 + \frac12a_{1, 3}a_{1, 4} + \frac14a_{1, 4}^2
\end{align*}
as the first entry of our signature matrix. The remaining $3$ polynomial entries of $S(X)$  are provided in \Cref{ex:closed_formk2}. We illustrate the polynomial membrane \eqref{eq:polynomialMembran_intro} for an explicit matrix $A$ in \Cref{fig:illustrativeMembranes}.

\subsection*{Main results}

The varieties of polynomial  signature matrices $\mathcal{M}_{d,m,n}$ can be arranged in the following 
 grid of inclusions,  
\begin{equation}\label{eq:grid_of_varieties}
    \begin{matrix}
        \mathcal{M}_{d,1,1} & \subseteq & \mathcal{M}_{d,1,2} & \subseteq & \mathcal{M}_{d,1,3}& \subseteq & \dots & \subseteq &\mathcal{M}_{d,1,N}\\
        \rotatebox[origin=c]{270}{$\subseteq$} 
        &&\rotatebox[origin=c]{270}{$\subseteq$}
        &&\rotatebox[origin=c]{270}{$\subseteq$}
        &&&&
        \rotatebox[origin=c]{270}{$\subseteq$}\\
        \mathcal{M}_{d,2,1} & \subseteq & \mathcal{M}_{d,2,2} & \subseteq & \mathcal{M}_{d,2,3}& \subseteq & \dots & \subseteq &\mathcal{M}_{d,2,N}\\
        \rotatebox[origin=c]{270}{$\subseteq$} 
        &&\rotatebox[origin=c]{270}{$\subseteq$}
        &&\rotatebox[origin=c]{270}{$\subseteq$}
        &&&&
        \rotatebox[origin=c]{270}{$\subseteq$}\\
        \vdots &  & \vdots &  & \vdots&  & \ddots & &\vdots\\
        \rotatebox[origin=c]{270}{$\subseteq$} 
        &&\rotatebox[origin=c]{270}{$\subseteq$}
        &&\rotatebox[origin=c]{270}{$\subseteq$}
        &&&&
        \rotatebox[origin=c]{270}{$\subseteq$}\\
         \mathcal{M}_{d,M,1} & \subseteq & \mathcal{M}_{d,M,2} & \subseteq & \mathcal{M}_{d,M,3}& \subseteq & \dots & \subseteq &\mathcal{M}_{d,M,N}
    \end{matrix}
    \end{equation}
    where we can choose $M,N\in\mathbb{N}$ such that
    \begin{align*}
\mathcal{M}_{d,m,1}\subseteq\dots\subseteq\mathcal{M}_{d,m,N-1}\subseteq\mathcal{M}_{d,m,N}=\mathcal{M}_{d,m,N+1}&=\dots\quad
        \\
        \mathcal{M}_{d,1,n}\subseteq\dots\subseteq\mathcal{M}_{d,M-1,n}\subseteq\mathcal{M}_{d,M,n}=\mathcal{M}_{d,M+1,n}&=\dots
    \end{align*}
    for all  $m,n\in\mathbb{N}$.
  We show in \Cref{lmm:membrane vs path signature} that  $\mathcal{M}_{d,1,m} = \mathcal{M}_{d,m,1}$ agrees with the path varieties studied in \cite{bib:AFS2019}. In particular, it 
follows that $\mathcal{M}_{d,1,N}=\mathcal{M}_{d,M,1}$ is the variety parametrizing the signature matrices of all
smooth paths, called \emph{universal path variety} in \cite[Section 4.3]{bib:AFS2019}. However, in sharp contrast to the path case, there are generally no algebraic relations in our two-parameter signature matrices \eqref{eq:membrane_sig_tensor} at all. 
  
\begin{theorem*}[\ref{thm:mainResult}]
    The variety $\mc M_{d,M,N}$ in \eqref{eq:grid_of_varieties} attains the full ambient dimension, i.e. $\mathcal{M}_{d,M,N}  = \mb C^{d\times d}$. More precisely, $\mc M_{d,m,n} = \mb C^{d\times d}$ whenever $m + n > d$ with $m,n\not=1$. 
\end{theorem*}
The proof relies on the theory of matrix congruence orbits \cite{TERAN201144,HORN20061010}. Using this machinery, we can find lower and upper bounds for our varieties $\mc M_{d,m,n}$: 
\begin{theorem*}[{\ref{thm:rks of sk and sym}}]
Let $\mc S_{a,b}$ be the variety of $d\times d$ matrices whose symmetric part has rank $\leq a$ and whose skew-symmetric part has rank $\leq b$. Then
\begin{itemize}
    \item $\mc S_{(m-2)(n-2), m+n-2} \subseteq \mc M_{d,m,n} \subseteq \mc S_{mn,m+n-2}$ if $m,n$ are even,
    \item $\mc S_{2, n-1} \subseteq \mc M_{d,2,n} \subseteq \mc S_{2(n-1)+1, n + 1}$ if $n$ is odd,
    \item $\mc S_{(m-2)(n-1) - 1, m+n-1} \subseteq \mc M_{d,m,n} \subseteq \mc S_{m(n-1)+1, m+n-1}$ if $m\geq 4$ is even, $n$ is odd,
    \item $\mc S_{(m-1)(n-1) + 1,m+n-2} = \mc M_{d,m,n}$ if $m,n$ are odd.
\end{itemize}
\end{theorem*}
In particular, it follows that our varieties are only interesting if $m$ and $n$ are small compared to $d$. From \cite{TERAN201144} we obtain explicit formulas for the dimensions of $\mathcal{M}_{d,m,n}$  whenever $mn\leq d$, generalising \cite[Theorem 3.4]{bib:AFS2019}. Note that there are $dmn \leq d^2$ parameters involved, so the dimension can be at most $dmn$.

\begin{theorem*}[\ref{thm:pw-lin matr dim}]
   For $mn \leq d$, the dimension of $\mc M_{d,m,n}$ is
    \begin{itemize}
       \item $dmn - \frac{1}{2}m^2 n^2 +  m^2 (n-1) + (m-1)n^2 - \frac 7 2 mn + 4(m+n) - 4$, if $m$ and $n$ even,
       \item $dmn- \frac{1}{2}m^2 n^2+m^2(n-1) + (m-1)n^2 - \frac 3 2 mn + m + n$, if $m$ even and $n$ odd,
       \item $dmn - \frac{1}{2}m^2 n^2 + m^2 (n-1) + (m-1)n^2 - \frac 7 2 mn + 3(m+n) - 2$, if $m$ and $n$ odd.
    \end{itemize}
\end{theorem*}
The results are complemented by some explicit computations using the open-source computer algebra system \textsc{Macaulay 2} \cite{GraysonStillmanMacaulay2,eisenbud2001computations}, see e.g. \Cref{ex:berndspoly,ex:eliminationConcrete_dkmn}, \Cref{rem:explicitCompForMissingDims} and \Cref{ex:bensPolyWith100Terms}.
\subsection*{Piecewise bilinear membrances}
Bilinear interpolation is an easy way to fit discrete data into the framework of iterated integrals. It allows us to \emph{interpolate} arbitrary $d$-dimensional \emph{discrete data} on an $m\times n$ grid via a \emph{piecewise bilinear membrane of order} $(m,n)$; see \Cref{def:pw bilin}.  For an illustration of such an interpolation see $Y$ from \Cref{fig:illustrativeMembranes}. In \Cref{def:def_axis_membrane} we define our second important dictionary, the \textit{axis membrane} $\mathsf{Axis}^{m,n}$. It is inspired by the dictionary for piecewise linear paths from \cite{bib:PSS2019}, and it allows an analogue of \cite[Theorem 3.3]{bib:AFS2019} for membranes. 

\begin{theorem*}[\ref{thm:matrixVarEq}]
The sets of signature matrices of polynomial and piecewise bilinear membranes with order $(m,n)$ coincide.
In particular, both sets define the same algebraic variety $\mathcal{M}_{d,m,n}$.
\end{theorem*}

A naive computation of signature matrices via congruence operation results only in a quadratic time algorithm with respect to the number of interpolation points; see \Cref{thm:complx}. In particular, there is no analogue of Chen's identity for the id-signature \cite{GLNO2022,diehl2024signature}. However, using a cumulated approach inspired by \cite{DEFT22,DS22} we can drop this cost to linear complexity. 
\begin{corollary*}[\ref{cor:complexitySigMatOfBilMembrane}]
    The signature matrix of a piecewise bilinear membrane in $\mb R^d$ of order $(m,n)$ is computable in $\mathcal{O}(d^2 m n)$.
\end{corollary*}
We can generalize this idea for the $k$-th level of the $\mathsf{id}$-signature tensor; see  \Cref{thm:complexitySigOfBilMembrane}. This yields a linear two-parameter integration algorithm for the signatures of images.\par
The code accompanying this paper is available at 
\begin{center}
    \url{https://github.com/felixlotter/signature-matrices-of-membranes}.
\end{center}

\subsection*{Related work}
The collection of iterated integrals of a path has been used as a tool for classification in various settings since Chen's seminal 1954 work \cite{Chen1954IteratedIA}. 
After 
pioneering work by Lyons and collaborators on the use of signatures as features for data streams in the early 2010s \cite{Gyu13,handw13}, there has been a surge in the number of applications.

Since \cite{bib:AFS2019,bib:PSS2019,GALUPPI2019282} established the initial connection between stochastic analysis and algebraic geometry, 
the signature has also been applied in various related fields, including the theory of tensor decompositions \cite{AGOSS2025,FLS24,galuppi2024ranksymmetriessignaturetensors},  clustering tasks \cite{clausel2024barycenterfreenilpotentlie,mignot:hal-04392568}, tropical time series analysis \cite{DEFT22,diehl2023fruitsfeatureextractionusing} 
and the geometry of paths \cite{Diehl2019,10.1007/978-3-031-38271-0_45,preiß2024algebraicgeometrypathsiteratedintegrals}. 

Recently, the signature has been extended to two parameters, generalizing from paths to membranes: the \emph{mapping space signature} \cite{GLNO2022} provides a principled two-parameter extension of Chen’s iterated integrals, based on its original use in topology, whereas the \emph{two-parameter sums signature} \cite{DS22} characterizes equivalence up to warping in two directions. These two approaches have recently been combined in \cite{diehl2024signature} using mixed partial derivatives.  
Due to the lack of a two-parameter version of Chen's identity in all the above references, recent works \cite{CDEFT2024, LO2023, lee2024} use higher category theory to find an appropriate generalization of a group that is rich enough to represent the horizontal and vertical concatenation of membranes.

\paragraph{Outline.}
In section \Cref{subsec:coordinateFreeIdSig} we recall the notion of iterated-integrals signatures for paths and membranes.
\Cref{sec:prod_membranes} introduces the notion of a product membrane. For arbitrary $k$, the $k$-th level signature of these membranes can be computed from path signatures, which will allow us to use theory from the path setting. We define the moment membrane and the canonical axis membrane and include algorithmic considerations of bilinear interpolations and linearly computable signature entries in \Cref{sec:subsectionBilinear}.
Finally, we return to the matrix case $k=2$ in \Cref{sec:sig_mat} where we prove our main results on signature matrix varieties. For open problems and future work see \Cref{sec:outlook}.

\section{Iterated integrals}\label{subsec:coordinateFreeIdSig}
Let $V := \mathbb R^d$ and let $T(V^*)$ denote the tensor algebra $\bigoplus_{k \in \mathbb N_0} (V^*)^{\tensor k}$. The iterated-integrals signature $\sigma(X)$ of a $d$-dimen\-sional path $X:[0,1] \to V$ is an element of the tensor series algebra $T((V^*))\cong \Hom(T(V), \mb R)$, defined by
\begin{align}\label{eq:one_param_sig}
    \sigma(X): T(V^*) &\to \mb R, \nonumber\\
    e_{i_1}^* \tensor \dots \tensor e_{i_k}^* &\mapsto \int_{\Delta_k} \mathrm{d}X_{i_1} (t_1) \dots \mathrm{d}X_{i_k}(t_k),
\end{align}
where $\Delta_k$ is the simplex $0 \leq t_1 \leq \dots \leq t_k \leq 1$. The signature matrix of $X$ is $S_{ij}(X) := \sigma(X)(e_i^* \tensor e_j^*)$. It is common to write the tensor $e^*_{i_1}\tensor \dots \tensor e^*_{i_k}$ as a word $\mathtt{i}_1\dots \mathtt{i}_k$, identifying the tensor algebra $T(V^*)$ with the free associative algebra $\mathbb R \langle \mathtt{1},\ldots,\mathtt{d} \rangle$ over the letters $\texttt{1}$ to $\texttt{d}$. 
However, the tensor algebra viewpoint allows us to consider the signature for paths (and later membranes) in any $\mb R$-vector space $V$, without choosing a basis, which will be useful in the following. Let us first recall some examples in the (classical, one-parameter) path setting. 

\begin{example}\label{example_one_prameter_signatures}
\begin{enumerate}
    \item For fixed $u\in\mathbb{R}^d$, the $k$-th signature tensor of the \DEF{linear path} $X: [0,1]\rightarrow\mathbb{R}^d,t\mapsto ut$ is given by \begin{equation}\label{eq:sig_lin_path}
        \sigma(X)(e_{i_1}^*\otimes\dots\otimes e_{i_k}^*) =\frac{u_{i_1}\dots u_{i_k}}{k!}
    \end{equation}
    where $1\leq i_1,\dots,i_k\leq d$.
    \item Let  $\mathsf{Axis}^d: [0,1] \to \mb R^d$ be the \DEF{canonical axis path} 
    from  \cite[Example 2.1]{bib:AFS2019} or \cite[eq. (8)]{bib:PSS2019}  of order $d$, i.e., the unique piecewise linear path with $d$ segments whose graph interpolates the $d+1$ equidistant support points $(\frac{i}{d}, \sum_{1\leq j\leq i}e_j)$ with $0 \leq i \leq d$. This path is illustrated in \Cref{fig:axis3_path} for $d=3$. For all $d$, its signature matrix is the upper triangular matrix
    \begin{equation}
        S_{i_1,i_2}(\mathsf{Axis}^d)=\begin{cases}
            1&\text{if }i_1<i_2\\
            \frac12&\text{if }i_1=i_2\\
            0&\text{if }i_1>i_2. 
        \end{cases}
    \end{equation}
    \label{eq:signatureMatrixAxisPath}
   \item  We denote by $\mathsf{Mom}^d: [0,1] \to \mb R^d,t \mapsto (t, \dots, t^d)$ the \textit{moment path} of degree $d$. We have the closed formula 
   \begin{equation}\label{eq:momPathSignature}\sigma(\mathsf{Mom}^d)(e_{i_1}^*\otimes\dots\otimes e_{i_k}^*)=\frac{i_2i_3\dots i_k}{(i_1+i_2)(i_1+i_2+i_3)\dots (i_1+\dots+i_k)} 
   \end{equation} 
 for all $i_1,\dots,i_k$ due to  \cite[Example 2.2]{bib:AFS2019}. 
   \end{enumerate}
\end{example}

\begin{figure}
    \centering
    \begin{subfigure}{0.24\textwidth}
        \centering
\includegraphics{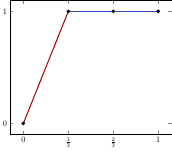}
\caption*{$\mathsf{graph}(\mathsf{Axis}^{3}_{1})$}
\end{subfigure}
\hfill
        \begin{subfigure}{0.24\textwidth}
        \centering
\includegraphics{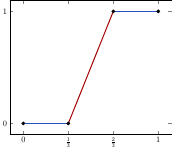}
\caption*{$\mathsf{graph}(\mathsf{Axis}^{3}_{2})$}
\end{subfigure}
\hfill
\begin{subfigure}{0.24\textwidth}
        \centering
\includegraphics{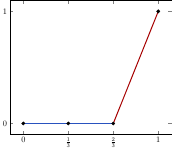}
\caption*{$\mathsf{graph}(\mathsf{Axis}^{3}_{3})$}
\end{subfigure}
\hfill
\begin{subfigure}{0.24\textwidth}
        \centering
\includegraphics{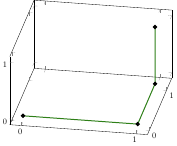}
\caption*{$\mathsf{Im}(\mathsf{Axis}^{3}$)}
\end{subfigure}

    \caption{The $3$ graphs of the coordinate functions $\mathsf{Axis}^3_i$ with the non-constant part marked in red, and the entire image of the $3$-dimensional path $\mathsf{Axis}^3$ on the right. Compare also \Cref{fig:axis32} for the correspondending axis membrane.}
    \label{fig:axis3_path}
\end{figure}

The path signature \eqref{eq:one_param_sig} can be extended to \textit{two parameters} in multiple ways. We will study the $\mathsf{id}$-signature as defined in \cite[Definition 4.1]{diehl2024signature}. We comment on its various modifications in \Cref{rem:sig_versions}. 

Let $V$ be a finite dimensional $\mathbb R$-vector space. A \textit{membrane} in $V$ is a Lipschitz continuous map $X: [0,1]^2 \to V$.

\begin{definition}[{$\mathsf{id}$-signature}]\label{def:signature}
Let $V$ be a finite dimensional vector space and $X$ a membrane in $V$. Let $\partial_{12}:=\frac {\partial^2} {\partial x_1 \partial x_2}$ be the linear partial differential operator. Then, since both the integral and $\partial_{12}$ are linear operators, the following linear form $\sigma(X)\in T(V^*)^*$ is well-defined:
\begin{align*}
    \sigma(X): T(V^*) &\to \mb R, \\
    \alpha_1 \tensor \dots \tensor \alpha_k &\mapsto \int_{(\mathbf{s},\mathbf{t}) \in \Delta^2_k} \partial_{12}(\alpha_1 \circ X) (s_1,t_1) \dots \partial_{12}(\alpha_k\circ X)(s_k,t_k) \,\mathrm{d}\mathbf{s} \mathrm{d}\mathbf{t}
\end{align*}
where $\mathrm{d}\mathbf{s}=\mathrm{d}s_{1}\dots \mathrm{d}s_{k}$, $\mathrm{d}\mathbf{t}=\mathrm{d}t_{1}\dots \mathrm{d}t_{k}$ and $\Delta_k^2=\Delta_k\times\Delta_k$. We call $\sigma(X)$ the \textit{$\mathsf{id}$-signature} of $X$. The \DEF{$k$-th level} signature $\sigma^{(k)}(X)$ is the restriction of $\sigma(X)$ along $(V^*)^{\tensor k} \to T(V^*)$, in particular it is naturally an element of $ V^{\tensor k} \cong ((V^*)^{\tensor k})^*$.
\end{definition}

\begin{remark}\label{rem:sig_versions}
\begin{enumerate}
        \item\label{rem:sig_versions2} When `replacing' our square increments $\partial_{12}$ with \emph{Jacobian minors}  we obtain the \emph{$\mathsf{id}$-signature} introduced in \cite[Section 4.7]{GLNO2022}.
    \item\label{rem:sig_versions3} For membranes and standard basis $e_1^*,\dots,e_d^*$ of $V^*=(\mathbb{R}^d)^*$, we obtain the \emph{$\mathsf{id}$-signature} according to \cite{diehl2024signature}. Its level $k$ component can be stored in the $d\times \dots \times d$ array, 
    \begin{align*}\sigma^{(k)}_{i_1,\dots, i_k}(X) &:=\sigma(e_{i_1}^* \tensor \dots \tensor e_{i_k}^*)\\
    &=\int\limits_{\substack{0\leq s_1\leq\dots\leq s_k\leq1\\0\leq t_1\leq\dots\leq t_k\leq1}}\partial_{12}X_{i_1}(s_1,t_1)\dots \partial_{12}X_{i_k}(s_k,t_k)\,\mathrm{d}s_{1}\mathrm{d}t_{1} \dots \mathrm{d}s_{k}\mathrm{d}t_{k}. 
    \end{align*}
    In particular, for level $k=2$, we obtain our signature matrix $S(X)=\sigma^{(2)}(X)$ from the introduction, \eqref{eq:membrane_sig_tensor}.
    \item When `replacing' integrals by sums, and $\partial_{12}$ by its discrete counterpart, we obtain the two-parameter \emph{sums signature} \cite{DS22}, evaluated at all \emph{diagonal matrix compositions}. For this compare also \cite[Example 4.16]{diehl2024signature}. 
\end{enumerate}
\end{remark}

\begin{remark}\label{rem:reduced membrane}
    It is clear from the definition that the signatures of any two membranes $X,Y:[0,1]^2\to V$ agree if $\partial_{12}(\alpha\circ (X-Y))=0$ for all $\alpha \in V^*$. In particular, if $X:[0,1]^2 \to V$ is a membrane, we can always find a membrane $X_\mathrm{red}$ which vanishes on $[0,1] \times \{0\} \cup \{0\} \times [0,1]$ and has the same signature as $X$, by letting $X_\mathrm{red}(s,t) := X(s,t) - X(0,t) - X(s,0) + X(0,0)$.
    See also  \Cref{lem:obdaZeroOnBound} and \Cref{fig:axis32} for an illustration of piecewise bilinear membranes.  
\end{remark}

We provide some first examples for membranes. 

\begin{example}\label{ex:sig_bilin_membrane}
\begin{enumerate}
    \item\label{ex:sig_bilin_membrane1}
    For fixed $u\in\mathbb{R}^d$, the $k$-th signature tensor of the $d$-dimen\-sional \DEF{bilinear membrane} $X: [0,1]^2\rightarrow\mathbb{R}^d,(s,t)\mapsto ust$ is given by 
    \begin{equation}\label{eq:sig_bilin_membrane}
        \sigma^{(k)}_{i_1,\dots,i_k}(X)=\frac{u_{i_1}\dots u_{i_k}}{(k!)^2}
    \end{equation}
    where $1\leq i_1,\dots,i_k\leq d$. 
This formula follows from \Cref{cor:scaling path membrane} or can be verified by an immediate computation.
    \item\label{ex:sig_bilin_membrane2} We recall the $2$-dimensional polynomial membrane $X=A\,\mathsf{Mom}^{2,2}$ from the introduction,  \eqref{eq:polynomialMembran_intro}. 

With integration over the differentials of the two-parameter moment coordinates, 
$$\int_0^1\int_0^1\,\mathrm{d}s\mathrm{d}t
=\int_0^1\int_0^12t\,\mathrm{d}s\mathrm{d}t
=\int_0^1\int_0^12s\,\mathrm{d}s\mathrm{d}t
=\int_0^1\int_0^14st\,\mathrm{d}s\mathrm{d}t
=1,$$
we obtain, together with linearity, the level $1$ signature, 
$$\sigma^{(1)}(X)=\begin{pmatrix}
a_{1,1}+a_{1,2}+a_{1,3}+a_{1,4}\\
a_{2,1}+a_{2,2}+a_{2,3}+a_{2,4}
\end{pmatrix}\in\mathbb{R}^2.$$
In \Cref{ex:closed_formk2} and \ref{ex:closed_form_k3} we provide closed formulas for its level $2$ and $3$ signature tensors, respectively. 
    \end{enumerate}
\end{example}

As in the previous example, we can use linearity of the integral to compute level $k$ signatures of transformed functions by corresponding actions on tensors. 

\begin{lemma}[Equivariance]\label{lmm:coord-free equivariance}
    Let $V, \ W$ be finite dimensional $\mb R$-vector spaces, $X: [0,1]^2 \to V$ a membrane and $A: V \to W$ a linear map. Then $\sigma(AX) = \sigma(X) \circ T(A^*)$ where $T(A^*)$ denotes the map $T(W^*) \to T(V^*)$ induced by diagonal action of $A^*$ on tensors.
\end{lemma}
\begin{proof}
    This follows immediately from the definition, since for $w=w_1 \tensor \dots \tensor w_k$ (where $w_i \in W^*$) we have $w_i \circ (A \circ X) = (w_i \circ A) \circ X = A^* (w_i) \circ X$.
\end{proof}

In particular, \Cref{lmm:coord-free equivariance} implies that $\sigma^{(k)}(AX) = A^{\tensor k}(\sigma^{(k)}(X))$ for all $k$. If $V=\mathbb{R}^d$ we can write this in the \emph{Tucker format} $\sigma^{(k)}(AX) = [\![\sigma^{(k)}(X); A, \dots, A]\!]$. Explicitly, if $k=2$ and $S(X)$ denotes the signature matrix, we have $$S(AX)=AS(X)A^\top,$$
    as was already mentioned in the introduction \eqref{eq:intro_matrix_congruence}. 
The equivariance of signatures is a key property used in \cite{bib:PSS2019}, allowing the computation of signature tensors for various families of paths from a small collection of \textit{core tensors}. These are given by signature tensors of special paths called \textit{dictionaries}.
In the following section we derive dictionaries for membranes from such dictionaries for paths.

\section{Product membranes}\label{sec:prod_membranes}
Some membranes arise from combining paths: given a path $[0,1]\to V$ and another path $[0,1] \to W$ we can consider the Cartesian product $[0,1]^2 \to V \times W$. Composing this with any suitable function $\phi: V\times W \to Z$ into another vector space $Z$ yields a membrane. The first goal of this section is to show that if $\phi$ is \emph{bilinear}, then we can express the signature of the resulting membrane in terms of the signatures of the defining paths. This is particularly important for us, as we will see that both polynomial and piecewise bilinear membranes can be obtained in this way. Recalling equivariance of the signature (\Cref{lmm:coord-free equivariance}) and that any bilinear map $V \times W \to Z$ factors over a linear map $V \tensor W \to Z$, it suffices to understand the signature of the following type of membrane:

\begin{definition}\label{def:product_membranes} Let $V$ and $W$ be finite dimensional vector spaces. 
    The \DEF{product membrane} of two paths $X: [0,1] \to V$ and $Y: [0,1] \to W$ is the membrane in the tensor product $V \tensor W$ induced by the Cartesian product: $$X \boxtimes Y: [0,1]^2 \to V \times W \to V \tensor W.$$In other words, $X \boxtimes Y(s,t) := X(s)\tensor Y(t)$ for all $(s,t)\in[0,1]^2$. 
\end{definition}

\begin{example}
\begin{enumerate}
    \item 
    Consider the path $X$ mapping $s \mapsto (s,s^2)$ and the path $Y$ mapping $t \mapsto (e^t,\log(1+t))$. Writing elements of $\mb R^2 \tensor \mb R^2$ as matrices, the product membrane $X \boxtimes Y$ is given by
    $$(s,t) \mapsto \matr{se^t & s \log(1+t) \\
                            s^2e^t & s^2 \log(1+t)}\in\mb{R}^{2\times 2}.$$
        \item The
       \emph{moment membrane of order  $(m,n)$} in the introduction \eqref{eq:into_momentMembrane}  is the product 
    \begin{align}\label{def:moment_membrane}
        \mathsf{Mom}^{m,n}=\mathsf{Mom}^m \boxtimes \mathsf{Mom}^n: [0,1]^2 &\to \mb R^{m} \tensor \mb R^n, 
    \end{align}
    where the moment path $\mathsf{Mom}^m$ is according to   \Cref{example_one_prameter_signatures}. 
    Note that we sometimes omit the last vectorization, i.e., we identify $\mathbb{R}^{m\times n}$ with $\mathbb{R}^{mn}$ whenever the context allows this. For this compare also \Cref{cor:coordinate_product_membrane_kron}. 
    Note that every $d$-dimensional  polynomial membrane $X$ of order $(m,n)$ and with terms of order $\geq (1,1)$ is of the form $X=A\circ\mathsf{Mom}^{m,n}$, where $A$ is a linear map, e.g. our running \Cref{ex:sig_bilin_membrane} with $m=n=d=2$. 
\end{enumerate}
\end{example}

\subsection{Factorizations into path signatures}\label{subsec:productsOfPaths}

We state our main result of this section, \Cref{prop:product membranes}, showing that the signature is \emph{multiplicative} with respect to products of paths (\Cref{def:product_membranes}). In \Cref{cor:coordinate_product_membrane_kron}, we rephrase this for the matrix case  $k=2$, with explicit vectorizations and the use of the Kronecker matrix product. We illustrate these results in Examples \ref{ex:closed_form_k3} and \ref{ex:closed_formk2},  respectively.

\begin{proposition}\label{prop:product membranes}
Let $X \boxtimes Y: [0,1]^2 \to V \tensor W$ be a product membrane. Then, as linear maps,
$$\sigma(X \boxtimes Y) = \sigma(X) \tensor \sigma(Y)$$
in the following sense: the map $\sigma(X \boxtimes Y)$ factors as
$$T((V \tensor W)^*) \leftisomto T(V^* \tensor W^*) \to T(V^*) \tensor T(W^*) \to \mb R,$$
where the second map is induced by the universal property of the tensor algebra, and the last map is given by $\sigma(X) \tensor \sigma(Y)$.
More explicitly, for all $k \geq 1$, $\alpha_i \in V^*$, $\beta_i \in W^*$ we have
\begin{align*}
    \left<\sigma(X \boxtimes Y),(\alpha_1\tensor \beta_1)\tensor \dots \tensor (\alpha_k\tensor \beta_k)\right> &= \\
    \left<\sigma(X), \alpha_1 \tensor \dots \tensor \alpha_k\right> &\cdot \left<\sigma(Y),\beta_1 \tensor \dots \tensor \beta_k\right>.
\end{align*}
\end{proposition}
\begin{caveat}
   Note that the tensor product of linear forms $\sigma(X): T(V^*)\to \mb R$, $\sigma(Y): T(V^*)\to \mb R$ is not the tensor product in the tensor series space $T(V^*)^*$; the formula above should not be confused with Chen's formula for the signature of a concatenation of paths.
\end{caveat}
Before proving \Cref{prop:product membranes}, we provide examples and discuss several implications.
As a first application, we compute the signature of product membranes where one path is $1$-dimensional, explaining for instance our example \eqref{eq:sig_bilin_membrane} using path signatures. 
\begin{corollary}\label{cor:scaling path membrane}
    Let $X:[0,1] \to \mb R$ be a $1$-dimensional path, and  $Y:[0,1] \to W$ be arbitrary. Then $\sigma(X \boxtimes Y)$ is a map $T(\mb R \tensor W^*) \to \mb R$, so we can view it as a map $T(W^*) \to \mb R$. Then it sends a simple tensor $w = w_1 \tensor \dots \tensor w_\ell$ to 
    $$\sigma(X)(e_1^{\tensor\ell}) \cdot \sigma(Y)(w) = \frac{(X(1)-X(0))^{\ell}}{\ell!}\cdot \sigma(Y)(w).$$
    In particular, $\sigma^{(k)}(X \boxtimes Y) = \frac{1}{k!} (X(1)-X(0))^{k} \cdot \sigma^{(k)}(Y)$, so  $\sigma^{(k)}(X\boxtimes Y)$ and $\sigma^{(k)}(Y)$ differ only by some scalar for any $k$.
\end{corollary}
\begin{proof}
    Unravel the factorization of $\sigma(X \boxtimes Y)$ in \Cref{prop:product membranes}.
\end{proof}

If $X:[0,1]^2\to V \tensor W$ is a product membrane, it can be convenient to use a vectorization $\nu: V \tensor W \cong \mb R^{\dim(V)\cdot\dim(W)}$, for example to obtain an actual signature matrix. Using \Cref{lmm:coord-free equivariance}, we see that $\sigma(\nu X)$ then factors as
$$T(\mb R^{\dim(V)\cdot\dim(W)}) \to T((V \tensor W)^*) \to T(V^*) \tensor T(W^*) \to \mb R$$
where the first map is $T(\nu^*)$.\par
As a special case, we obtain compatibility of the signature with the Kronecker product of paths:
\begin{corollary}\label{cor:coordinate_product_membrane_kron}
    Let $X:[0,1]\to \mb R^m$ and $Y:[0,1] \to \mb R^n$ be two paths and let $\mathsf{kron}(X,Y):[0,1]^2 \to \mb R^{mn}$ denote the membrane that maps $(s,t)$ to the Kronecker product of the column vectors $X(s)$ and $Y(t)$. Then we have the following identity between signature matrices:
    \begin{align*}S(\mathsf{kron}(X, Y)) 
&= S(X)\tensor S(Y)
\end{align*}
where $\tensor$ denotes the Kronecker product of matrices.
\end{corollary}\label{xmpl:vect membranes}

\begin{example}\label{ex:closed_formk2}
We return to our running \Cref{ex:sig_bilin_membrane} with $X=A\,\mathsf{Mom}^{2,2}$
and matrix $A=(a_{i,j})_{1\leq i\leq2,1\leq j \leq 4}$. 
With 
\Cref{cor:coordinate_product_membrane_kron} and \eqref{eq:momPathSignature} we get the core matrix of the moment membrane, 
$$C:=\sigma^{(2)}(\mathsf{Mom}^{2,2})
=\sigma^{(2)}(\mathsf{Mom}^{2})\otimes
\sigma^{(2)}(\mathsf{Mom}^{2})
=\begin{pmatrix}
    \frac12&\frac23\\[4pt]
    \frac13&\frac12
\end{pmatrix}\otimes\begin{pmatrix}
    \frac12&\frac23\\[4pt]
    \frac13&\frac12
\end{pmatrix}$$
as it is presented in the introduction, \eqref{eq:ex_C2}. 
With   
$S(X)=\sigma^{(2)}(X)=A C A^\top$ we provide the three missing homogeneous entries, 
\begin{align*}
    S(X)_{2,1}
    &=\frac{1}{4} a_{1, 1} a_{2, 1} + \frac{1}{6} a_{1, 1}a_{2, 2} + \frac{1}{6} a_{1, 1} a_{2, 3} + \frac{1}{9} a_{1, 1} a_{2, 4} + \frac{1}{3} a_{2, 1} a_{1, 2} + \frac{1}{3} a_{2, 1} a_{1, 3} \\ &\quad+ \frac{4}{9} a_{2, 1} a_{1, 4}+ \frac{1}{4} a_{1, 2} a_{2, 2} + \frac{2}{9} a_{1, 2} a_{2, 3} + \frac{1}{6} a_{1, 2} a_{2, 4} + \frac{2}{9} a_{2, 2} a_{1, 3} + \frac{1}{3} a_{2, 2} a_{1, 4}  \\ &\quad + \frac{1}{4} a_{1, 3} a_{2, 3}+ \frac{1}{6} a_{1, 3} a_{2, 4}+ \frac{1}{3} a_{2, 3} a_{1, 4} + \frac{1}{4} a_{1, 4} a_{2, 4}\\
    S(X)_{1,2}&=
    \frac{1}{4} a_{1, 1} a_{2, 1} + \frac{1}{3} a_{1, 1} a_{2, 2} + \frac{1}{3} a_{1, 1} a_{2, 3} + \frac{4}{9} a_{1, 1} a_{2, 4} + \frac{1}{6} a_{2, 1} a_{1, 2} + \frac{1}{6} a_{2, 1} a_{1, 3} \\ &\quad+ \frac{1}{9} a_{2, 1} a_{1, 4} + \frac{1}{4} a_{1, 2} a_{2, 2} + \frac{2}{9} a_{1, 2} a_{2, 3} + \frac{1}{3} a_{1, 2} a_{2, 4} + \frac{2}{9} a_{2, 2} a_{1, 3} + \frac{1}{6} a_{2, 2} a_{1, 4} \\ &\quad+ \frac{1}{4} a_{1, 3} a_{2, 3} + \frac{1}{3} a_{1, 3} a_{2, 4} + \frac{1}{6} a_{2, 3} a_{1, 4} + \frac{1}{4} a_{1, 4} a_{2, 4}
\intertext{\begin{normalsize}and\end{normalsize}}
    S(X)_{2,2}&=\frac{1}{4} a_{2, 1}^{2} + \frac{1}{2} a_{2, 1} a_{2, 2} + \frac{1}{2} a_{2, 1} a_{2, 3} + \frac{5}{9} a_{2, 1} a_{2, 4} + \frac{1}{4} a_{2, 2}^{2} + \frac{4}{9} a_{2, 2} a_{2, 3} + \frac{1}{2} a_{2, 2} a_{2, 4} \\&\quad+ \frac{1}{4} a_{2, 3}^{2} + \frac{1}{2} a_{2, 3} a_{2, 4} + \frac{1}{4} a_{2, 4}^{2}
\end{align*}
using equivariance,  \Cref{lmm:coord-free equivariance}. The polynomial $S(X)_{1,1}$ is already provided in the introduction. \par
We illustrate the membrane $X$ in the first row of \Cref{fig:illustrativeMembranes}, for the explicit choice  $A=\begin{psmallmatrix}1&-1&1&1\\1&1&0&-1\end{psmallmatrix}$. In particular, we provide the graphs of the coordinate functions $\mathsf{graph}(X_1)$ and $\mathsf{graph}(X_2)$, and a plot of the image $\mathsf{Im}(X)$ in $\mb R^2$. Here, the colors correspond to the image of rectangles in a uniform partition of $[0,1]^2$, e.g. $X([0,\frac13]\times [0,\frac12])$ in pink, $X([\frac13,\frac23]\times [0,\frac12])$ in orange, or $X([\frac23,1]\times [\frac12,1])$ in violet. Note that $X$ maps three sides of the square $[0,1]^2$ to the origin. The right-most picture shows an ``untangled" version of the membrane $X$ which has the same signature, cf. \Cref{rem:reduced membrane}.
\end{example}

\begin{figure}[h]
    \centering
    \begin{subfigure}{0.24\textwidth}
        \centering
        \includegraphics{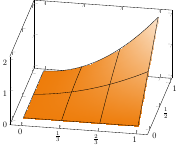}
\caption*{$\mathsf{graph}(X_1)$}
\end{subfigure}
\hfill
        \begin{subfigure}{0.24\textwidth}
        \centering
\includegraphics{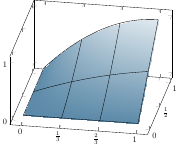}
\caption*{$\mathsf{graph}(X_2)$}
\end{subfigure}
\hfill
\begin{subfigure}{0.24\textwidth}
        \centering
    \includegraphics{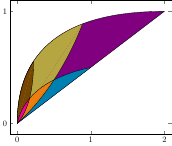}
\caption*{$\mathsf{Im}(X)$}
\end{subfigure}
\hfill
\begin{subfigure}{0.24\textwidth}
        \centering
         \includegraphics{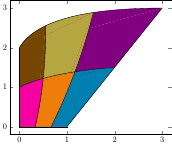}

\caption*{$\mathsf{Im}(X+\begin{psmallmatrix}
    s\\2t
\end{psmallmatrix})$}
\end{subfigure}
\vspace{0.2cm}
            \centering
            \vspace{0.2cm}
    \begin{subfigure}{0.24\textwidth}
        \centering
        \includegraphics{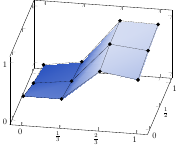}
\caption*{$\mathsf{graph}(Y_1)$}
\end{subfigure}
\hfill
        \begin{subfigure}{0.24\textwidth}
        \centering
         \includegraphics{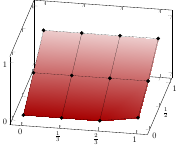}
\caption*{$\mathsf{graph}(Y_2)$}
\end{subfigure}
\hfill
\begin{subfigure}{0.24\textwidth}
        \centering
         \includegraphics{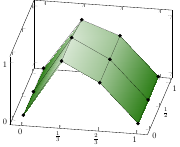}
\caption*{$\mathsf{graph}(Y_3)$}
\end{subfigure}
\hfill
\begin{subfigure}{0.24\textwidth}
        \centering
        \includegraphics{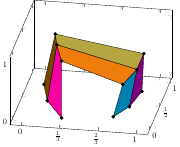}
\caption*{$\mathsf{Im}(Y)$}
\end{subfigure}

    \caption{Plots of the polynomial membrane $X$ in $\mb R^2$ from \Cref{ex:closed_formk2} and a piecewise bilinear membrane $Y$ interpolating $12$ points in $\mathbb{R}^3$, represented by black dots. Let $X_1,X_2, Y_1, Y_2$ and $Y_3$ denote the coordinate functions of $X$ and $Y$, respectively.}
    \label{fig:illustrativeMembranes}
\end{figure}

\begin{proof}[Proof of \Cref{cor:coordinate_product_membrane_kron}]
     Choosing the vectorization $\nu: \mb R^m \tensor \mb R^n \to \mb R^{mn},e_i \tensor e_j \mapsto e_{n(i-1) + j}$, we see that $\mathsf{kron}(X,Y) = \nu \circ X \boxtimes Y$. Then we have
    \begin{align*}
    & \left<\sigma(\nu (X \boxtimes Y)),e_i \tensor e_j \right> = \left<\sigma(X \boxtimes Y),\nu^*(e_i) \tensor \nu^*(e_j)\right>\\
    &= \left<\sigma(X \boxtimes Y),(e_{i_1} \tensor e_{i_2}) \tensor (e_{j_1} \tensor e_{j_2}) \right> =
    \left<\sigma(X), e_{i_1} \tensor e_{j_1} \right> \cdot \left<\sigma(Y),e_{i_2}  \tensor e_{j_2}\right>
\end{align*}
where $i_1 =\lceil \frac {i} n \rceil$, $j_1 = \lceil \frac j n \rceil$, $i_2 = (i - 1 \text{ mod } n) + 1$, $j_2 = (j - 1 \text{ mod } n) + 1$. In terms of signature matrices, this amounts to the identity
\begin{equation*}
    \sigma^{(2)}(\nu (X \boxtimes Y)) = \sigma^{(2)}(X) \tensor \sigma^{(2)}(Y).
\end{equation*}
 
\end{proof}

As another consequence of \Cref{prop:product membranes} and equivariance we obtain the formula from \cite[Example 4.2]{diehl2024signature} for the Hadamard product of two paths:

\begin{corollary}\label{cor:evaluatedProductMembranes}
Let $\mu: \mb R^d \tensor \mb R^d \to \mb R^d$ denote the componentwise multiplication map. 
If $V=W=\mb R^d$ then 
\begin{equation}\label{eq:product_formula}
        \sigma(\mu(X\boxtimes Y)) = \sigma(X) \cdot \sigma(Y)
    \end{equation}
    as morphisms $T(\mb R^d) \to \mb R$. In particular, $S(X \odot Y) = S(X) \odot S(Y)$ where $\odot$ denotes the Hadamard product of vectors resp.\ matrices.
\end{corollary}
\begin{proof}
    The dual map $\mu^*$ sends $e_i$ to $e_i \tensor e_i$.
\end{proof}

We close this section with our omitted proof, followed by an example for higher tensor levels, and an implication for the computational complexity of the signature of a polynomial membrane.  

\begin{proof}[Proof of \Cref{prop:product membranes}]
The isomorphism $T(V^* \tensor W^*)\to T((V\tensor W)^*)$ displayed in the statement is induced by the map that sends $\alpha \tensor \beta$ to the form $\alpha \cdot \beta$ on $V\tensor W$. Now, by definition, $\sigma(X)$ maps a simple tensor $\alpha_1 \cdot \beta_1 \tensor \dots \tensor \alpha_k \cdot \beta_k$ to
$$\int_{\Delta} \int_{\Delta} \prod_i \partial_{12} z_i(s_i,t_i) \,\mathrm{d}s_1 \dots \mathrm{d}s_k \mathrm{d}t_1 \dots \mathrm{d}t_k$$
where $z_i(s,t) := (\alpha_i \circ X)(s)\cdot (\beta_i \circ Y)(t)$. Here we reordered the differential forms which is allowed because we are integrating over the product simplex $\Delta^2$. In particular, we have $(\partial_{12} z_i)(s,t) = \partial(\alpha_i \circ X)(s) \cdot \partial(\beta_i \circ Y)(t)$. By definition,
$$\langle\sigma(X^{(1)}),\alpha_1 \tensor \dots \tensor \alpha_k\rangle = \int_\Delta \prod_i \partial(\alpha_i \circ X)(s_i) \,\mathrm{d}s_1 \dots \mathrm{d}s_k$$
and similarly for $Y$; thus reordering the integral reveals the claimed equality.
\end{proof}


In order to provide an explicit computation for higher tensor levels $k\geq 3$, we introduce some further notation. 
\begin{rmknot}\label{rem:notation_tuple_of_words}
If $X:[0,1] \to \mb R^m$ and $Y:[0,1]\to \mb R^n$ are two paths, then by \Cref{prop:product membranes} the signature of $X \boxtimes Y$ is a linear form on $T((\mb R^m \tensor \mb R^n)^*)$ which can be identified with the free associative algebra $\mb R\langle(i,j)\rangle$ over tuples of letters $i\in[m]$ and $j\in[n]$. Thus, in the following we also write $(i_1,j_1)\dots(i_k,j_k)$ for the tensor $(e_{i_1}\tensor e_{j_1})^* \tensor \dots \tensor (e_{i_k}\tensor e_{j_k})^*$.
\end{rmknot}

\begin{example}\label{ex:closed_form_k3}
  With \Cref{prop:product membranes}, our identification in   \Cref{rem:notation_tuple_of_words}, and the closed formula for signatures of polynomial paths in \Cref{example_one_prameter_signatures}, we obtain a closed formula for the level $k$ tensor of the moment membrane $\mathsf{Mom}^{m,n}$ from \eqref{def:moment_membrane},\begin{align*}
  \langle\sigma(\mathsf{Mom}^{m,n})&,(i_1,j_1)\dots(i_k,j_k)\rangle\\&=\frac{i_2\dots i_k}{(i_1+i_2)\dots (i_1+\dots+i_k)} \frac{j_2\dots j_k}{(j_1+j_2)\dots (j_1+\dots+j_k)}
   \end{align*}
    where $ i_1,\dots,i_k\in[m]$ and $j_1,\dots,j_k\in[ n]$. 
For instance for the $2$-dimensional polynomial membrane from our running example \Cref{ex:sig_bilin_membrane}, we obtain the level $k=3$ core tensor, here denoted by $C^{(3)}$,
$$
\begin{scriptsize}
\left(\begin{array}{cccc|cccc|cccc|cccc}
 \frac1{36}  & \frac1{36} &  \frac1{36}   &\frac1{36}
 &\frac1{24} & \frac2{45} &\frac1{24} & \frac2{45}
 &\frac1{24} & \frac1{24} & \frac2{45}  &\frac2{45}
 &\frac1{16}  & \frac1{15}  & \frac1{15}&   \frac{16}{225}
 \\[4pt]
 \frac1{72}  & \frac1{60} & \frac1{72}   &\frac1{60}
 &\frac1{45}&  \frac1{36} & \frac1{45}  &\frac1{36}
 &\frac1{48}& \frac1{40}&  \frac1{45} & \frac2{75}
 &\frac1{30} & \frac1{24} & \frac{8}{225} & \frac{2}{45}
 \\[4pt]
 \frac1{72}  & \frac1{72}   &\frac1{60}   &\frac1{60}
 &\frac1{48}  &\frac1{45} & \frac1{40}  &\frac2{75}
 &\frac1{45} & \frac1{45} & \frac1{36} &\frac1{36}
 &\frac1{30} &  \frac8{225} & \frac1{24}&   \frac2{45}
 \\[4pt]
 \frac1{144} & \frac1{120} & \frac1{120}& \frac1{100}
 &\frac1{90}  &\frac1{72} & \frac1{75}  &\frac1{60}
 &\frac1{90} & \frac1{75}  &\frac1{72} & \frac1{60}
 &\frac4{225} & \frac1{45} &  \frac1{45}&   \frac1{36}
 \end{array}
 \right)
 \end{scriptsize}
 $$
which is of shape $4\times 4\times 4$. With equivariance (\Cref{lmm:coord-free equivariance}), this results in the level $k=3$ signature tensor  $\sigma^{(3)}(X)=[\![C^{(3)}; A, A, A]\!]$ of shape 
 $2\times 2\times 2$. We provide its \emph{first} homogeneous entry, 
 \begin{align*}
 {\sigma^{(3)}(X)}_{1,1,1}&=\frac{1}{36} a_{1, 1}^{3} + \frac{1}{12} a_{1, 1}^{2} a_{1, 2} + \frac{1}{12} a_{1, 1}^{2} a_{1, 3} + \frac{7}{72} a_{1, 1}^{2} a_{1, 4} + \frac{1}{12} a_{1, 1} a_{1, 2}^{2} \\&\quad+ \frac{11}{72} a_{1, 1} a_{1, 2} a_{1, 3} + \frac{13}{72} a_{1, 1} a_{1, 2} a_{1, 4} + \frac{1}{12} a_{1, 1} a_{1, 3}^{2} + \frac{13}{72} a_{1, 1} a_{1, 3} a_{1, 4} \\&\quad+ \frac{89}{900} a_{1, 1} a_{1, 4}^{2} + \frac{1}{36} a_{1, 2}^{3} + \frac{5}{72} a_{1, 2}^{2} a_{1, 3} + \frac{1}{12} a_{1, 2}^{2} a_{1, 4} \\&\quad+ \frac{5}{72} a_{1, 2} a_{1, 3}^{2} + \frac{34}{225} a_{1, 2} a_{1, 3} a_{1, 4} + \frac{1}{12} a_{1, 2} a_{1, 4}^{2} + \frac{1}{36} a_{1, 3}^{3} \\&\quad+ \frac{1}{12} a_{1, 3}^{2} a_{1, 4} + \frac{1}{12} a_{1, 3} a_{1, 4}^{2} + \frac{1}{36} a_{1, 4}^{3}. 
 \end{align*}
 The other (polynomial) entries of $\sigma^{(3)}(X)$ are similar and are therefore omitted here.  
\end{example}
\begin{corollary}\label{thm:complx}
    For every $d$-dimensional polynomial membrane $X$ of order $(m,n)$ we can compute $\sigma^{(k)}(X)$ in $\mathcal{O}(m^kn^k+kd^{k+1}mn)$ elementary operations. 
\end{corollary}
\begin{remark}\label{rem:complx}
   We can generalize  \Cref{thm:complx} for any class of membranes which are induced via linear transforms of product membranes, e.g., piecewise bilinear membranes from the next section.  
   Note that for the latter we can improve the analogous computational complexity bound by \Cref{thm:complexitySigOfBilMembrane}. 
\end{remark}

\subsection{Dictionaries for bilinear interpolation}\label{sec:subsectionBilinear}

\begin{definition}\label{def:pw bilin}
    A \DEF{piecewise bilinear membrane} is a membrane $X$ such that there are some $m,n \in \mb N$ with the property that $X$ is biaffine on all rectangles $[\frac{i-1}{m},\frac{i}{m}] \times [\frac{j-1}{n},\frac{j}{n}]$ with  $i \in [m]$ and $j \in[n]$. We say that $X$ has order  $(m,n)$.
\end{definition} 
Our interest in this class of membranes is threefold: from the path case, examined in \cite{bib:AFS2019}, we expect a certain relation of this class of membranes to the polynomial membranes in terms of their signatures, e.g. \Cref{thm:matrixVarEq}. We also know that the \emph{recovery problem} is solvable for piecewise linear paths of low order from \cite[Corollary 6.3]{bib:PSS2019}, and we would like to obtain a similar result for membranes (\Cref{q:recoveryProblem}). Finally, they form a rather natural class that can be used for the interpolation of discrete data, e.g. image data used in \cite{ZLT22}.

Bilinear interpolation of $d$-dimensional discrete data, arranged on an $(m+1)\times (n+1)$ grid, is a well-known procedure, e.g. in image processing \cite[Chapter 3.7.2]{bovik2009essential} or in the field of finite elements \cite[Chapters 5.1.3.2 and 6.2.2.1]{zienkiewicz_finite_2005}, \cite[Chapter 1.3.2]{elman2014finite}. 
The approach involves considering the four (unique) \emph{nodal bilinear basis functions} on the \emph{reference rectangle}, where each function is $1$ at one corner and $0$ at all others. The nodal basis on any arbitrary rectangle is then obtained through a linear transformation to the reference rectangle. Bilinear interpolation is simply the linear combination of the nodal basis functions, where the coefficients correspond exactly to the function values of the function being interpolated at the four grid points.

As in the polynomial case \eqref{def:moment_membrane}  we provide a product dictionary for piecewise bilinear membranes, so that its signature can be computed via \Cref{prop:product membranes}. For this, we have already recalled in   \Cref{example_one_prameter_signatures} the canonical axis path $\mathsf{Axis}^m$, which is illustrated in \Cref{fig:axis3_path} for $m=3$.
The resulting axis membrane takes the role of the above mentioned bases for interpolation. 
It is illustrated for order $(3,2)$  in \Cref{fig:axis32}.

    \begin{definition}\label{def:def_axis_membrane}
        The (canonical) \DEF{axis membrane} of order $(m,n)$ is the product
    \begin{align}\label{eq:def_axis_membrane}
        \mathsf{Axis}^{m,n}:=\mathsf{Axis}^m \boxtimes \mathsf{Axis}^n: [0,1]^2 &\to \mb R^{m} \tensor \mb R^n,
    \end{align}
    according to \Cref{def:product_membranes}.
    \end{definition} 
     Any piecewise linear path with $m$ segments \textit{starting at the origin} can be realized through a suitable linear transformation of $\mathsf{Axis}^m$. We want to show that the analogue is true for membranes, i.e., \eqref{eq:def_axis_membrane} is a suitable dictionary for all piecewise bilinear membranes of order $(m,n)$ which \emph{map the coordinate axes $[0,1] \times 0 \cup 0 \times [0,1]$ to the origin}.
   Let us first describe the axis membrane a bit more explicitly.

\begin{figure}
    \centering
\begin{subfigure}{0.24\textwidth}
        \centering
\includegraphics{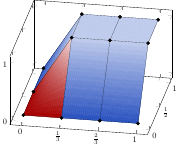}
\caption*{$\mathsf{Axis}^{3,2}_{1,1}$}
\end{subfigure}
\hfill
    \begin{subfigure}{0.24\textwidth}
        \centering
\includegraphics{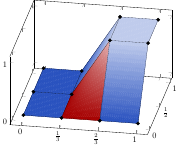}
\caption*{$\mathsf{Axis}^{3,2}_{2,1}$}
\end{subfigure}
\hfill
    \begin{subfigure}{0.24\textwidth}
        \centering
        \includegraphics{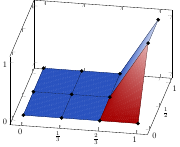}
        \caption*{$\mathsf{Axis}^{3,2}_{3,1}$} 
\end{subfigure}
\hfill 
\begin{subfigure}{0.24\textwidth}
        \centering
\includegraphics{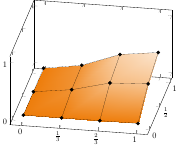}
\caption*{$A\circ\mathsf{Axis}_{3,2}$}
\end{subfigure}
\vspace{0.2cm}
            \centering
            \vspace{0.2cm}
\begin{subfigure}{0.24\textwidth}
        \centering
\includegraphics{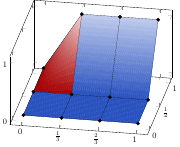}
\caption*{$\mathsf{Axis}^{3,2}_{1,2}$}
\end{subfigure}
\hfill
    \begin{subfigure}{0.24\textwidth}
        \centering
\includegraphics{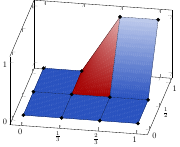}
\caption*{$\mathsf{Axis}^{3,2}_{2,2}$}
\end{subfigure}
\hfill
    \begin{subfigure}{0.24\textwidth}
        \centering
            \includegraphics{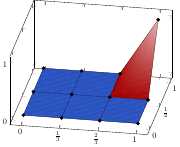}
        \caption*{$\mathsf{Axis}^{3,2}_{3,2}$} 
\end{subfigure}
\hfill 
\begin{subfigure}{0.24\textwidth}
        \centering
        \includegraphics{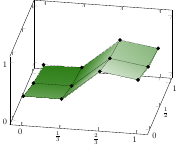}
\caption*{$R$}
\end{subfigure}
    \caption{On the left we have the $6$ graphs of the coordinate functions from the axis membrane $\mathsf{Axis}^{3,2}$, with the non-linear rectangle marked in red,  and the equidistant support points from \eqref{eq:linearSystemForAxisMembrane} in black. On the right we illustrate the decomposition of the $1$-dimensional piecewise bilinear membrane $Y_1$ from \Cref{fig:illustrativeMembranes} into the linear transform of the axis membrane $A\circ\mathsf{Axis}^{3,2}$ and the linear part $R$ according to \Cref{lem:obdaZeroOnBound}.}
    \label{fig:axis32}
\end{figure}

\begin{lemma}\label{lem:explicitversion}
 We consider the codomain from the axis membrane  \eqref{eq:def_axis_membrane} as a space of matrices, that is $\mathsf{Axis}^{m,n}:[0,1]^2\rightarrow\mathbb{R}^{m\times n}$. Then each of its coordinate functions is piecewise bilinear, 
$$\mathsf{Axis}^{m,n}_{i,j}(s,t):=
    \begin{scriptsize}
    \begin{cases}
        0
        &
        s\leq\frac{i-1}m\text{ or }t\leq\frac{j-1}n
        \\
        mnst+m(1-j)s+n(1-i)t+(i-1)(j-1)
        &
        \frac{i-1}m<s\leq\frac{i}m\text{ and }\frac{j-1}n<t\leq\frac{j}n
                \\
        nt+1-j
        &
        \frac{i}m<s\text{ and }\frac{j-1}n<t\leq\frac{j}n        \\
        ms+1-i
        &
        \frac{i-1}m<s\leq\frac{i}m\text{ and }\frac{j}n<t
        \\
        1
        &
        \frac{i}m<s\text{ and }\frac{j}n<t,
    \end{cases}
\end{scriptsize}
$$
where $i\in[m]$, $j\in[n]$ and $s,t\in [0,1]$. 
\end{lemma}
\begin{proof}
Despite this being well-known we provide a bilinear interpolation in each of the $mn$ dimensions, i.e., we search for coefficients $c\in\mathbb{R}^{4}$ such that 
\begin{align*}
\psi_{i,j}:\left[\frac{i-1}m,\frac{i}m\right]\times\left[\frac{j-1}n,\frac{j}n\right]&\rightarrow\mathbb{R},(s,t)\mapsto c_{1}st+c_{2}s+c_{3}t+c_{4}
\end{align*}
satisfies the system 
\begin{align}\label{eq:linearSystemForAxisMembrane}
   \psi_{i,j}\left(\frac{i-1}{m},\frac{j-1}{n}\right)&=0,\quad
   \psi_{i,j}\left(\frac{i}{m},\frac{j-1}{n}\right)=0,\nonumber\\
   \psi_{i,j}\left(\frac{i-1}{m},\frac{j}{n}\right)&=0,\quad
   \psi_{i,j}\left(\frac{i}{m},\frac{j}{n}\right)=1. 
\end{align}
By solving \eqref{eq:linearSystemForAxisMembrane} in $c$ we obtain  $c=\left(mn,m(1-j),n(1-i),(i-1)(j-1)\right)$ and hence $\psi=\mathsf{Axis}^{m,n}$ when restricted on the rectangle $\left[\frac{i-1}m,\frac{i}m\right]\times\left[\frac{j-1}n,\frac{j}n\right]$. The claim follows with a similar (linear) system for the remaining rectangle. 
\end{proof}

With the product structure in \eqref{eq:def_axis_membrane} we can use \Cref{prop:product membranes} to compute its core signature tensor with the help of the core signature tensors of axis paths. 
 \begin{example}\label{ex:closed_form_k3_axis}
 From \cite[Example 2.1]{bib:AFS2019} we know that $\sigma(\mathsf{Axis}^m)_{i_1,\dots,i_k}$ is zero unless $i_1 \leq \dots\leq i_k$, and  $\frac{1}{k!}$ times the number of distinct permutations of the string $i_1\dots i_k$ otherwise. 
Explicitly, with \eqref{eq:signatureMatrixAxisPath} the  signature matrix is 
$$\langle\sigma(\mathsf{Axis}^{m,n}),(i_1,j_1)(i_2,j_2)\rangle:=\begin{cases}1&i_1<i_2\text{ and }j_1<j_2\\
\frac12&i_1=i_2\text{ and }j_1<j_2\\
\frac12&i_1<i_2\text{ and }j_1=j_2\\
\frac14&i_1=i_2\text{ and }j_1=j_2\\
0&\text{otherwise.}
\end{cases}
$$
\end{example}

Recall from \Cref{rem:reduced membrane} that we can always restrict our computations to membranes which are $0$ on the coordinate axes.

\begin{lemma}\label{lem:obdaZeroOnBound}
    Every piecewise bilinear membrane $X$ has a (unique) decomposition \begin{equation}\label{eq:obdaZeroOnBound}X=A\circ\Axis^{m,n}+R
    \end{equation}
    where $\Axis^{m,n}$ is according to \eqref{eq:def_axis_membrane}, $A$ is a linear transform, and $R$ is a piecewise \emph{linear} (not bilinear) membrane in the two arguments $s$ and $t$. Then,  
    $$\sigma(X)=\sigma(A\circ\Axis^{m,n}).$$
\end{lemma}
\begin{proof}
    The coordinate
functions $\{\mathsf{Axis}^{m,n}_{i,j}\mid i\in\{1,\dots,m\},j\in\{1,\dots,n\}\}$ of the canonical axis membrane form
an $\mathbb{R}$-basis of the space of $1$-dimensional piecewise bilinear membranes. We can extend this basis via $m+n-1$ functions which are piecewise linear in $s$ and constant in $t$, or vice versa. Applying a basis decomposition in each dimension, we obtain $A$ and $R$ in \eqref{eq:obdaZeroOnBound} for all $d$ and conclude by \Cref{rem:reduced membrane}.
\end{proof}
The decomposition is illustrated in \Cref{fig:axis32}. See also \cite[Lemma 2.22]{DS22} for a similar result applied on discrete data, using two-parameter sums signatures. Inspired by \cite[Algorithm 1]{DEFT22} and \cite[Theorem 4.5]{DS22}, we obtain the following computational result. 

\begin{theorem}\label{thm:complexitySigOfBilMembrane}
    Let $X$ be a piecewise bilinear membrane in $\mb R^d$ of order $(m,n)$ and $w \in T(\mb R^d)$ a word of length $k$. Then we can compute $\langle \sigma(X),w \rangle$ in $\mathcal{O}(k^3 m n)$. In particular, the $k$-th signature tensor $\sigma^{(k)}(X)$ is computable in $\mathcal{O}(d^kk^3 m n)$. 
\end{theorem}
\begin{proof}
        For $a,b \in \mb N$, write $I_{a,b}$ for the rectangle $[\frac{a-1} m, \frac{a} m] \times [\frac{b-1} n, \frac{b} n]$.
        For a word $w = i_1 \ldots i_j$ and $(u,v) \in [0,1]^2$ define
        $$f_w(u,v) := \int_{\Delta_j(u,v)} \partial_{12} X_{i_1}(s_1,t_1) \ldots \partial_{12} X_{i_j}(s_j,t_j) \ \mathrm{d}t_1 \ldots \mathrm{d}t_j \mathrm{d}s_1 \ldots \mathrm{d}s_j$$
        where $\Delta_j(u,v)$ denotes the set
        $$\{ 0\leq s_1 \leq \dots \leq s_j \leq u  \} \times \{ 0\leq t_1 \leq \dots \leq t_j \leq v  \}.$$
        Set $f_w \equiv 1$ for $w$ the empty word.
        Then by definition, $f_w(1,1) = \langle \sigma(X), w \rangle$ and we have the recursion
        $$f_{wi}(u,v) = \int_0^v \int_0^u  f_w(s,t) \cdot \partial_{12} X_i(s,t) \ \mathrm{d}s \mathrm{d}t$$
        for any letter $i$. Thus it suffices to show that we can compute $f_{wi}$ from $f_w$ in time $\mc O(k^2mn)$. We can rewrite the integral above as the sum of the four terms
        $$\mathrm{Bl}_{wi}(u,v) := \sum_{a=1}^{\lfloor vm\rfloor} \sum_{b=1}^{\lfloor un\rfloor} \int_{\frac {a-1} m}^{\frac{a} m} \int_{\frac {b-1} n}^{\frac{b} n} f_w(s,t) \cdot \partial_{12} X_i(s,t) \ \mathrm{d}s \mathrm{d}t,$$
        $$\mathrm U_{wi}(u,v) := \sum_{a=1}^{\lfloor vm\rfloor} \int_{\frac {a-1} m}^{\frac{a} m} \int_{\frac {\lfloor un\rfloor} n}^{u} f_w(s,t) \cdot \partial_{12} X_i(s,t) \ \mathrm{d}s \mathrm{d}t,$$
        $$\mathrm R_{wi}(u,v) := \sum_{b=1}^{\lfloor un \rfloor} \int_{\frac {b-1} n}^{\frac{b} n} \int_{\frac {\lfloor vm \rfloor} m}^{v} f_w(s,t) \cdot \partial_{12} X_i(s,t) \ \mathrm{d}s \mathrm{d}t,$$
        and
        $$\mathrm{Ur}_{wi}(u,v) := \int_{\frac {\lfloor vm \rfloor} m}^v \int_{\frac {\lfloor un \rfloor} n}^ u f_w(s,t) \cdot \partial_{12} X_i(s,t) \ \mathrm{d}s \mathrm{d}t,$$
        see \Cref{fig:unitsquare}.
        \begin{figure}
            \centering
            \includegraphics[width=0.2\linewidth]{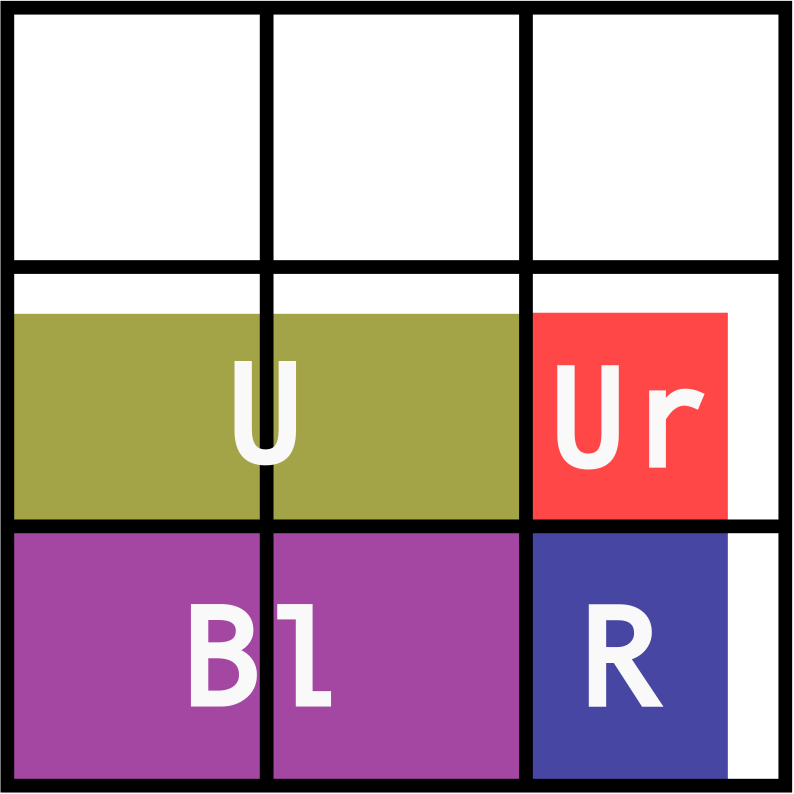}
            \caption{The four terms in the proof of \Cref{thm:complexitySigOfBilMembrane} correspond to the division of a rectangle in the $m \times n$ grid into four regions. In this example, $m=n=3$ and $(u,v)$ lies in $I_{3,2}$.}
            \label{fig:unitsquare}
        \end{figure}
        
        Note that if $f_w$ is given by a polynomial in $u$ and $v$ on all $I_{a,b}$, then again for all $a,b$ all four terms are polynomials in $u$ and $v$ for $(u,v) \in I_{a,b}$ and thus $f_{wi}$ will be a polynomial on all $I_{a,b}$ as well, since $\partial_{12} X_i$ is piecewise constant. As for the empty word $w = \emptyset$ we have $f_w \equiv 1$ it follows inductively that $f_w$ will always be polynomial of bidegree $(j,j)$ (where $j$ denotes the length of $w$) on each $I_{a,b}$. Thus, forming the at most $mn$ integrals in the definition of $\mathrm{Bl}_{wi}, \mathrm{U}_{wi}, \mathrm{R}_{wi}$ and $\mathrm{Ur}_{wi}$ takes time $\mc O(mnk^2)$. Then $f_{wi}$ can be computed via cumulative sums in $\mc O(mn)$ so that we get total runtime $\mc O(k^2mn)$, as desired.
\end{proof}

\begin{corollary}\label{cor:complexitySigMatOfBilMembrane}
    For every piecewise bilinear membrane  $X$ in $\mb R^d$ of order $(m,n)$ we can compute its signature matrix $S(X)=\sigma^{(2)}(X)$ in $\mathcal{O}(d^2 m n)$. 
\end{corollary}

\section{Signature matrix varieties}\label{sec:sig_mat}

For the remainder of this article, we focus on the second level signature, that is, on signature matrices  of membranes $S=\sigma^{(2)}$. Following the approach in \cite{bib:AFS2019} we restrict to families of $d$-dimensional membranes, 
\begin{equation}\label{eq:im_polyn_secVar}
\left\{S(X)\;\begin{array}{|l}X:[0,1]^2\rightarrow{\mathbb{R}}^d\text{ polynomial}\\
\text{with order}\leq(m,n)\end{array}\right\}
\end{equation}
and 
\begin{equation}\label{eq:im_bilinear_secVar}\left\{S(X)\;\begin{array}{|l}X:[0,1]^2\rightarrow{\mathbb{R}}^d\text{ piecewise bilinear}\\
\text{of order }\leq(m,n)\end{array}\right\}
\end{equation}
with fixed $m,n\in\mathbb{N}$. 
By \Cref{lmm:coord-free equivariance} and \Cref{lem:obdaZeroOnBound}, both \eqref{eq:im_polyn_secVar} and \eqref{eq:im_bilinear_secVar} can be considered as the image of a polynomial map of degree two,  
\begin{equation}\label{eq:def_rational_map}
S:\mathbb{R}^{d\times mn}\rightarrow\mathbb{R}^{d^2},A\mapsto ACA^\top. 
\end{equation}
Here  $C\in\mathbb{R}^{mn\times mn}$ is either the signature core matrix $S(\nu \mathsf{Mom}^{m,n})$ according to \Cref{ex:closed_form_k3} for polynomial membranes, or the axis core matrix $S(\nu \mathsf{Axis}^{m,n})$
for the piecewise bilinear family from \Cref{ex:closed_form_k3_axis}. We choose the vectorization $\nu: \mb R^m \tensor \mb R^n \to \mb R^{mn}$ as in the proof of \Cref{cor:coordinate_product_membrane_kron}. In particular, these signature images are semialgebraic subsets of $\mathbb{R}^{d\times d}$ by the Tarski-Seidenberg theorem (\cite[Theorem 2.2.1]{bochnak2013real}). For convenience, we omit $\nu$ from the notation in the following. Using our result on the signature of product membranes, we see that \cite[Theorem 3.3]{bib:AFS2019} generalizes immediately:

\begin{theorem}\label{thm:matrixVarEq} For all $d,m$ and $n$, the semialgebraic sets \eqref{eq:im_polyn_secVar} and \eqref{eq:im_bilinear_secVar} agree.
\end{theorem}
\begin{proof}
By definition of the two sets, it suffices to find an invertible matrix $A$ such that $S(\mathsf{Mom}^{m,n})= AS(\mathsf{Axis}^{m,n})A^\top$.
Now by \Cref{cor:coordinate_product_membrane_kron}, we obtain that  $S(\mathsf{Mom}^{m,n}) = S(\mathsf{Mom}^{m}) \tensor S(\mathsf{Mom}^{n})$ and $S(\mathsf{Axis}^{m,n}) = S(\mathsf{Axis}^{m}) \tensor S(\mathsf{Axis}^{n})$. Moreover, in the proof of \cite[Theorem 3.3]{bib:AFS2019} it is shown that for arbitrary $\ell$ there is an invertible $\ell \times \ell$ matrix $H_\ell$ such that $S(\mathsf{Mom}^\ell)(w) = H_\ell S(\mathsf{Axis}^\ell) H_\ell^\top$. Thus, we can choose $A:=H_m \tensor H_n$.
\end{proof}

We want to study the polynomials that vanish on this set, recorded in the homogeneous prime ideals $P_{d,m,n}$. Equivalently, we examine the tightest outer approximation by an algebraic variety. As it is common in applied algebraic geometry, we move to $\mathbb{C}$ and consider the \emph{Zariski closure} $\mathcal{M}_{d,m,n}$ of the image of $S_{\mathbb C} := \mathbb C \tensor S$
for arbitrary $d,m,n\in\mathbb{N}$. The graph of $S_{\mathbb C}$ is an algebraic variety with vanishing ideal 
$$I_{d,m,n}:=\langle(ACA^\top)_{ i,j}-s_{i,j}\mid 1\leq i,j\leq d\rangle.$$
Note that $I_{d,m,n}$  is an ideal in the polynomial ring of $mnd+d^2$ variables: the entries of $A$ and  $s_{i,j}$ for $1\leq i,j\leq d$. 
The projection to the target of \eqref{eq:def_rational_map} is then cut out by the prime ideal
$$P_{d,m,n}:=I_{d,m,n}\cap\mathbb{C}[s_{i,j}\mid 1\leq i,j\leq d],$$
and agrees with $\mc M_{d,m,n}$ by construction. Compare \Cref{ex:eliminationConcrete_dkmn} for an step-by-step construction with explicit choices of $d,m$ and $n$. As in the path framework, the level $k=2$ signature matrices of polynomial and piecewise bilinear membranes coincide.

\begin{example}\label{ex:eliminationConcrete_dkmn}  For explicit $d,m$ and $n$ we compute $P_{d,m,n}$ with the moment core matrix $C_{m,n}=S(\mathsf{Mom}^m)\otimes S(\mathsf{Mom}^n)$ and with \emph{Gr\"obner bases}, applying an \emph{elimination of variables} procedure. We refer to \cite{CLOS07} for an introduction to the latter. 
\begin{enumerate}
\item If $m=d=2$ and $n=1$, then $AC_{m,n}A^\top$ is 
    \begin{small}
    $$
    \begin{bmatrix}
        \frac{1}{4} a_{1, 1}^{2} + \frac{1}{2} a_{1, 1} a_{1, 2} + \frac{1}{4} a_{1, 2}^{2}
        &
        \!\!\!\!\frac{1}{4} a_{1, 1} a_{2, 1} + \frac{1}{3} a_{1, 1} a_{2, 2} + \frac{1}{6} a_{2, 1} a_{1, 2} + \frac{1}{4} a_{1, 2} a_{2, 2}
        \\
        \frac{1}{4} a_{1, 1} a_{2, 1} + \frac{1}{6} a_{1, 1} a_{2, 2} + \frac{1}{3} a_{2, 1} a_{1, 2} + \frac{1}{4} a_{1, 2} a_{2, 2}
        &
        \frac{1}{4} a_{2, 1}^{2} + \frac{1}{2} a_{2, 1} a_{2, 2} + \frac{1}{4} a_{2, 2}^{2}
    \end{bmatrix}
    $$
    \end{small}
    and with elimination of variables,
    \begin{equation}\label{eq:elimIdeal2121_poly}P_{2,2,1}=\langle 4s_{1, 1}s_{2, 2} - s_{2, 1}^2 - 2s_{2, 1}s_{1, 2} - s_{1, 2}^2\rangle\end{equation}
    is generated by the relation \cite[eq. (5)]{bib:AFS2019} in $\mathbb{C}^{2\times 2}$, i.e., the determinant of the symmetric part of the $2\times2$ matrix $S$.  
    \item 
    If $m=d=n=2$, then 
    \begin{equation}\label{eq:graphIdeal2221_poly}I_{2,2,1}=\langle S_{1,1}-s_{1,1},S_{2,1}-s_{2,1},S_{1,2}-s_{1,2},S_{2,2}-s_{2,2}\rangle\end{equation}
    with polynomials $S_{i,j}=S(X)_{i,j}$ according to \Cref{ex:closed_formk2}. With Gr\"obner bases we can show that 
    $P_{d,m,n}$ is the zero ideal, thus $\dim(\mathcal{M}_{2,2,2})=4$. 
    \end{enumerate}
\end{example}
Even for $m=n=2$, the last example illustrates that membranes generate \emph{larger} varieties than in the path framework \cite{bib:AFS2019}.
In fact, the dimension of $\mathcal{M}_{2,2,2}$ is maximal, providing an initial illustration of our main result,  \Cref{thm:mainResult}.

Let us make some easy observations about the varieties $\mc M_{d,m,n}$.

\begin{proposition}\label{prop:irred_Vars}
    The variety $\mc M_{d,m,n}$ is irreducible.
\end{proposition}
\begin{proof}
    By \Cref{lmm:coord-free equivariance}, $\mc M_{d,m,n}$ is the schematic image of  $A \mapsto A S(\mathsf{Axis}^{m,n}) A^\top$ where $A \in \mb C^{mn\times d}$. Thus, as $\mb C^{mn\times d}$ is irreducible, so is $\mc M_{d,m,n}$.
\end{proof}
\begin{corollary}For every 
 $d\in\mathbb{N}$ there exist $M,N\in\mathbb{N}$ such that $\mathcal{M}_{d,m,N+i}=\mathcal{M}_{d,m,N}$ and $\mathcal{M}_{d,M+i,n}=\mathcal{M}_{d,M,n}$ for all $i\in\mathbb{N}$. In other words, the chains in \eqref{eq:grid_of_varieties} stabilize.
\end{corollary}
Taking a product of an $m$-dimensional path with a $1$-dimensional path results in (scaled) path signature tensors. This immediately implies the following:

\begin{lemma}\label{lmm:membrane vs path signature}
    $\mc M_{d,m,1} = \mc M_{d,m}^{\mathsf{path}}$ where $\mc M^\mathsf{path}_{d,m}$ denotes (the affine cone of) the signature matrix variety considered in \cite[Section 3.1]{bib:AFS2019}.
\end{lemma}
\begin{proof}
     This follows from \Cref{cor:scaling path membrane} as we can always identify a map $A: \mb R \tensor \mb R^m \to \mb R^d$ with a map $A': \mb R^m \to \mb R^d$ such that $A = 
   \mathsf{id} \tensor A'$ and thus $A \circ \mathsf{Axis}^{m,1} = \left(A' \circ \mathsf{Axis}^m\right) \boxtimes \mathsf{Axis}^1$.
\end{proof}

The next lemma is technical, but holds the key to understanding the varieties $\mc M_{d,m,n}$.

\begin{lemma}\label{lmm:sig normal form}
    Let $C_{m,n}$ denote the signature matrix $S(\Axis_{m,n})$. Then the canonical form for congruence of $C_{m,n}$, as defined in \cite{HORN20061010}, is
    \begin{itemize}
        \item $\Gamma_3 \oplus H_4(1)^{\oplus \frac{m+n-4}{2}} \oplus \Gamma_1^{\oplus (m-2)(n-2)+1}$ if $m,n$ are even, 
        \item $\Gamma_2 \oplus H_4(1)^{\oplus \frac{n-1}{2}} \oplus H_2(-1)^{\oplus \frac{m-2}2} \oplus \Gamma_1^{\oplus (m-2)(n-1)}$ if $m$ is even, $n$ is odd, 
        \item $H_2(-1)^{\oplus \frac{m+n-2}{2}} \oplus \Gamma_1^{\oplus (m-1)(n-1)+1}$ if $m,n$ are odd. 
    \end{itemize}
     We refer to \cite{TERAN201144} for the definition of the matrices $J_k(0)$, $\Gamma_k$ and $H_{2k}(\mu)$.
\end{lemma}
\begin{proof}
    We use the following fact which is proven in \cite[page 1016]{HORN20061010}: let $A$ be a non-singular complex square matrix. Then the blocks in the canonical form for congruence of $A$ are in one-to-one correspondence with the blocks in the Jordan canonical form of the cosquare $A^{-\top} A$ in the following way: the blocks $\Gamma_k$ in the canonical form for congruence correspond to the blocks $J_k((-1)^{k+1})$ in the Jordan canonical form of $A^{-\top} A$ and the blocks $H_{2k}(\mu)$ correspond to pairs of blocks $J_k(\mu) \oplus J_k(\mu^{-1})$ (the blocks $J_k(0)$ do not appear in the non-singular case).\par
    As $C_{m,n}$ is indeed non-singular (by \Cref{ex:closed_form_k3_axis} we have $\det(C_{m,n}) = \frac{1}{4^{mn}}$), our strategy is thus to determine the Jordan canonical form of $C_{m,n}^{-\top} C_{m,n}$. This is simplified significantly by the following observation: writing $C_u := S(\mathsf{Axis}^u)$, we have  $C_{m,n} = C_m \tensor C_n$ by \Cref{prop:product membranes}, where $\tensor$ denotes the Kronecker product. In particular $$C_{m,n}^{-\top} C_{m,n} = (C_m \tensor C_n)^{-\top} (C_m \tensor C_n) = C_m^{-\top} C_m \tensor C_n^{-\top} C_n.$$ Moreover, as observed in the proof of \cite[Theorem 3.4]{bib:AFS2019}, the matrix $C_u$ is congruent to $N_u := \inlmatr{1 & 1 \\ -1 & 0} \oplus \tau^{\frac{u-2}{2}}$ if $u$ is even, and congruent to $N_u := \matr{1} \oplus \tau^{\frac{u-1}{2}}$ if $u$ is odd, where $\tau = \inlmatr{0 & 1 \\ -1 & 0}$. Thus, we may replace $C_{m,n}$ by $N_m \tensor N_n$, $C_m$ by $N_m$ and $C_n$ by $N_n$. These observations reduce the problem to the following two questions:
    \begin{enumerate}
        \item What are the Jordan canonical forms of $N_u^{-\top}N_u$?
        \item Given the Jordan canonical forms of two matrices, what is the Jordan canonical form of their Kronecker product?
    \end{enumerate}
    For the first question, note that  $\tau^{-\top} \tau = \inlmatr{-1 & 0 \\ 0 & -1} = J_1(-1)^{\oplus 2}$ and $$\matr{1 & 1 \\ -1 & 0}^{-\top} \matr{1 & 1 \\ -1 & 0} = \matr{-1 & 0 \\ -2 & -1}$$
    whose Jordan canonical form is $\inlmatr{-1 & 1 \\ 0 & -1}=J_2(-1)$.
    Thus, if $u$ is odd then the canonical form of $N_u^{-\top} N_u$ is $J_1(1) \oplus J_1(-1)^{\oplus u-1}$, and if $u$ is even then the canonical form of $N_u^{-\top} N_u$ is $J_2(-1) \oplus J_1(-1)^{\oplus u-2}$.\par
    The answer to the second question is \cite[Theorem 4.6]{BRUALDI198531} from which we obtain: the Jordan canonical form of $C_{m,n}^{-T} C_{m,n}$ is
    \begin{itemize}
        \item $J_3(1) \oplus J_1(1) \oplus J_2(1)^{\oplus m-2 + n-2} \oplus J_1(1)^{\oplus (m-2)(n-2)}$ if $m,n$ are even, 
        \item $J_2(-1) \oplus J_2(1)^{\oplus n-1} \oplus J_1(-1)^{\oplus m-2} \oplus J_1(1)^{\oplus (m-2)(n-1)}$ if $m$ is even, $n$ is odd, 
        \item $J_1(1) \oplus J_1(-1)^{\oplus m-1 + n-1} \oplus J_1(1)^{\oplus (m-1)(n-1)}$ if $m,n$ are odd. 
    \end{itemize}
    Thus the canonical form for congruence of $C_{m,n}$ is
    \begin{itemize}
        \item $\Gamma_3 \oplus H_4(1)^{\oplus \frac{m+n-4}{2}} \oplus \Gamma_1^{\oplus (m-2)(n-2)+1}$ if $m,n$ are even, 
        \item $\Gamma_2 \oplus H_4(1)^{\oplus \frac{n-1}{2}} \oplus H_2(-1)^{\oplus \frac{m-2}2} \oplus \Gamma_1^{\oplus (m-2)(n-1)}$ if $m$ is even, $n$ is odd, 
        \item $H_2(-1)^{\oplus \frac{m+n-2}{2}} \oplus \Gamma_1^{\oplus (m-1)(n-1)+1}$ if $m,n$ are odd. 
    \end{itemize}
\end{proof}

\begin{corollary}\label{thm:pw-lin matr dim}
    For $mn \leq d$ the dimension of $\mc M_{d,m,n}$ is
    \begin{itemize}
       \item $dmn - \frac{1}{2}m^2 n^2 +  m^2 (n-1) + (m-1)n^2 - \frac 7 2 mn + 4(m+n) - 4$ for even $m,n$.
       \item $dmn- \frac{1}{2}m^2 n^2+m^2(n-1) + (m-1)n^2 - \frac 3 2 mn + m + n$ for even $m$, odd $n$.
       \item $dmn - \frac{1}{2}m^2 n^2 + m^2 (n-1) + (m-1)n^2 - \frac 7 2 mn + 3(m+n) - 2$ for odd $m,n$.
    \end{itemize}
    We omit the case $m$ odd, $n$ even due to symmetry.
\end{corollary}
\begin{proof}
    This follows from \Cref{lmm:sig normal form} by using the codimension count in \cite[Theorem 2]{TERAN201144}.
\end{proof}

\begin{remark}
    In light of \Cref{lmm:membrane vs path signature}, one checks that for $n=1$ the polynomials in \Cref{thm:pw-lin matr dim} simplify to $md - \binom{m}{2}$, which is indeed the dimension of the affine cone of the path signature matrix varieties considered in \cite[Theorem 3.4]{bib:AFS2019}.
\end{remark}

\begin{corollary}\label{thm:rks of sk and sym}\label{rem:dim bound}
Let $\mc S_{a,b}$ be the variety of $d\times d$ matrices whose symmetric part has rank $\leq a$ and whose skew-symmetric part has rank $\leq b$. Then
\begin{enumerate}
    \item $\mc S_{(m-2)(n-2), m+n-2} \subseteq \mc M_{d,m,n} \subseteq \mc S_{mn,m+n-2}$ if $m,n$ are even,
    \item $\mc S_{2, n-1} \subseteq \mc M_{d,2,n} \subseteq \mc S_{2(n-1)+1, n + 1}$ if $n$ is odd,
    \item $\mc S_{(m-2)(n-1) - 1, m+n-1} \subseteq \mc M_{d,m,n} \subseteq \mc S_{m(n-1)+1, m+n-1}$ if $m\geq 4$ is even, $n$ is odd,
    \item $\mc S_{(m-1)(n-1) + 1,m+n-2} = \mc M_{d,m,n}$ if $m,n$ are odd.
\end{enumerate}
\end{corollary}
\begin{proof}
    For the upper bounds, we just need to compute the ranks of the symmetric and skew-symmetric parts of the normal forms in \Cref{lmm:sig normal form}.
    \begin{itemize}
        \item  For $m,n$ even the rank of $C_{m,n}^\mathsf{sym}$ is $(m-2)(n-2) + 1 + 2m + 2n - 8 + 3 = mn$. The rank of $C_{m,n}^\mathsf{sk}$ is $2 + m + n - 4 = m + n - 2$.
        \item For $m$ even, $n$ odd the rank of $C_{m,n}^\mathsf{sym}$ is $1 + 2(n-1) + (m-2)(n-1) = m(n-1) + 1$. The rank of $C_{m,n}^\mathsf{sk}$ is $2 + n - 1 + m - 2 = m + n - 1$.
        \item For $m, n$ odd the rank of $C_{m,n}^\mathsf{sym}$ is $(m-1)(n-1) + 1$. The rank of $C_{m,n}^\mathsf{sk}$ is $m + n - 2$.
    \end{itemize}
    For the lower bound, note that for a general matrix $X \in \mc S_{a,b}$ there are $S,T \in \GL_d(\mb C)$ with $S^\top X^{\mathsf{sym}} S = \Gamma_1^{\oplus a'}$ and $T^\top X^{\mathsf{sk}} T = H_2(-1)^{\oplus b'}$ where $a' \leq a$ and $b' \leq \frac b 2$. Thus, to give a lower bound, one can count blocks in the normal form of $C_{m,n}$ that are congruent to $H_2(-1)$ and can thus contribute to the skew-symmetric rank, such that all other blocks are congruent to $\Gamma_1$ and can contribute to the symmetric rank.\par
    For $m,n$ even, we have that $\Gamma_3 \oplus \Gamma_1$ is congruent to $H_2(-1)$, as is $H_4(1)$. Thus, every general skew-symmetric matrix of rank $2 + m + n - 4$ is congruent to $\Gamma_3 \oplus H_4(1)^{\oplus \frac{m+n -4} 2} \oplus \Gamma_1$. As every rank $(m-2)(n-2)$ symmetric matrix is congruent to $\Gamma_1^{\oplus (m-2)(n-2)}$ we conclude. \par
    For $m = 2$, $n$ odd, we have that $\Gamma_2 \oplus H_4(1)$ is congruent to $H_2(-1) \oplus \Gamma_1^{\oplus 2}$. Every general skew-symmetric matrix of rank $2 + n - 3 = n - 1$ is congruent to $H_2(-1) \oplus H_4(1)^{\oplus \frac{n-1} 2}$ and every symmetric matrix of rank $2$ is congruent to $\Gamma_1^{\oplus 2}$ and we conclude. \par
    For $m$ even, $n$ odd, and $m\geq 4$ we know that $\Gamma_2 \oplus \Gamma_1$ is congruent to $H_2(-1)$; thus every general skew-symmetric matrix of rank $2 + n - 1 + m - 2 = m+n-1$ is congruent to $\Gamma_2 \oplus H_4(1)^{\oplus \frac{n-1} 2} \oplus H_2(-1)^{\oplus \frac{m-2} 2} \oplus \Gamma_1$; and we conclude again.\par
    By the same reasoning, we immediately see that for $m,n$ odd we just obtain the variety $\mc S_{(m-1)(n-1) + 1,m+n-2}$ as a lower bound as well and deduce equality with $\mc M_{d,m,n}$.
\end{proof}

Note that for $m=1$ or $n=1$ we obtain a rank one symmetric part in accordance with the path framework,  \cite[eq. (4)]{bib:AFS2019}. Moreover, the corollary shows for all cases except $m=2$, $n$ odd, that the projection of matrices in $\mc M_{d,m,n}$ to the skew-symmetric part gives precisely the variety of skew-symmetric matrices of rank $\leq m+n-1$ (if $m+n$ is odd) or $\leq m+n-2$ (if $m+n$ is even). 


In the case where $m,n$ are odd, we can describe the variety precisely:
\begin{corollary}\label{cor:deg m n odd}
    The degree of $\mc M_{d,m,n}$ is
    $$\frac{1}{2^{d - (m + n) + 2}}
    \prod_{\alpha = 0}^{d-(m-1)(n-1)} \frac{ \binom{d + \alpha}{d- (m-1)(n-1) + 1 - \alpha}} {\binom{2\alpha+1}{\alpha}}
    \prod_{\alpha = 0}^{d-(m+n)} \frac{ \binom{d + \alpha}{d-(m+n) + 1 - \alpha}} {\binom{2\alpha+1}{\alpha}}.$$
\end{corollary}
\begin{proof}
    This follows from \Cref{thm:rks of sk and sym} since $\mc M_{d,m,n} = \mc S_{(m-1)(n-1) + 1, m + n -2}$. Indeed, for $m,n$ odd, $m + n \leq d$, $\mc S_{(m-1)(n-1) + 1, m + n -2}$ is the transverse intersection of the vanishing locus of the $((m-1)(n-1) + 2)$-minors of the symmetric part with the vanishing locus of the $(m + n)$-Pfaffians of the skew-symmetric part. The degree of these two vanishing loci is calculated in \cite[Proposition 12] {HARRIS198471}.
\end{proof}

\begin{example}
    According to \Cref{thm:pw-lin matr dim} and \Cref{cor:deg m n odd} the variety $\mc M_{6,3,3}$ has dimension $34$ and degree $18$. It is a complete intersection of the vanishing locus of the Pfaffian of the skew-symmetric part and the vanishing locus of the determinant of the symmetric part.
\end{example}
\begin{example}\label{ex:berndspoly}
    Consider the case $d=4$ and $m=n=2$. By \Cref{thm:rks of sk and sym}, we have $\mc S_{0, 2} \subseteq \mc M_{d,m,n} \subseteq \mc S_{4,2}$. In particular, the Pfaffian $x_{12}x_{34}-x_{13}x_{24} + x_{14}x_{23}$ vanishes on the skew-symmetric part of matrices in the congruence orbit. However, according to \Cref{thm:pw-lin matr dim} the dimension of $\mc M_{4,2,2}$ is 14, so there must be additional relations on the symmetric part. Indeed, another relation is given by $\det(A)\det(X^\mathsf{sym}) = \det(X) \det(A^\mathsf{sym})$: the ratio $\det(X)/\det(X^\mathsf{sym})$ is well-defined and clearly constant on the congruence orbit. Using $\textsc{Macaulay 2}$ one checks that the variety $\mc M_{4,2,2}$ is a complete intersection, cut out by those two equations.
\end{example}

It is immediate from \Cref{thm:rks of sk and sym} that $\mc M_{d,m,n}$ attains the full ambient dimension $d^2$ when $m,n \to \infty$. Our main theorem gives a precise bound:

\begin{theorem}\label{thm:mainResult}
    For $m,n \geq 2$, we have $\mc M_{d,m,n} = \mb C^{d\times d}$ whenever $m+n \geq d+1$.
\end{theorem}
\begin{proof}
     Clearly, it suffices to prove the statement for $d = m + n -1$. We make a case distinction for the different parities of $m$ and $n$. If $m$ and $n$ are odd, then we have $(m-1)(n-1) + 1 \geq m + n -1$. Moreover, $d = m + n - 1$ is odd, so that every $d\times d$ matrix has skew-symmetric rank $\leq m + n - 2$. Thus, we immediately conclude by \Cref{thm:rks of sk and sym}.\\
    We prove the statement for the remaining cases by induction on $m+n$. We settle the base cases by explicit linear algebra calculations. \par
    \textit{The base case.}
    Consider the map $\phi: \mb C^{mn \times d} \to \mb C^{d \times d}$ given by $A \mapsto A^\top C_{m,n} A$. Let us write $\phi_B'$ for its derivative $$\mb C^{mn \times d} \to \mb C^{d \times d}, A \mapsto B^\top C_{m,n} A + A^\top C_{m,n} B$$ at a matrix $B \in \mb C^{mn \times d}$. It suffices to give a matrix $B$ where $\phi_B'$ has full rank.\par
    The base case for $m, n$ even is $m=n=2$, $d = 3$. Then the canonical form of $A_{2,2}$ is $\Gamma_3 \oplus J_1(1)$. We compute that for $B = \begin{pmatrix}
        \mathrm{I}_d & \mathbbm{1}
    \end{pmatrix}$, $\phi_B'$ has full rank.\par
    The base case for $m$ even, $n$ odd is $m=2$, $n=3$, $d = 4$. We compute that for $B = \begin{pmatrix}
        \mathrm{I}_d & \mathbbm{1} & \mathbbm{1} 
    \end{pmatrix}$, $\phi_B'$ has full rank.\par
    \textit{The induction step.} Assume the statement is true for $m + n = d + 1$ for some $d$. We claim that it then also holds for $m + n = d + 3 = (d + 2) + 1$, i.e.\, $\mc M_{d+2,m,n} = \mb C^{d+2,d+2}$. By assumption, we have $\mc M_{d,m,n-2} = \mb C^{d \times d}$ or $\mc M_{d,m-2,n} = \mb C^{d \times d}$. A general $\mb (d+2) \times (d+2)$-matrix has normal form $\oplus_{i \leq k} H_2(\mu_i) \oplus \Gamma_1^{\oplus l}$ where $2 k + l = d + 2$ with $l \in \{0,1\}$ and $\mu_i \in \mb C$. In particular, wlog. we can replace a general $(d+2)\times (d+2)$ matrix by a matrix $D \oplus E$ for $D$ a general $d \times d$-matrix and $E$ a general $2 \times 2$-matrix.\par
    If $m$ and $n$ are even and wlog. $m \geq 4$, then we have $C_{m,n} = C_{m-2,n} \oplus H_4(1) \oplus \Gamma_1^{\oplus 2(n-2) + 1}$. A calculation in \textsc{Macaulay 2} shows that generically the fibers of the map $\mb C^{4 \times 2}\to \mb C^{2 \times 2}, A \mapsto A^\top H_4(1) A$ are $5$-dimensional.\par
    If $m$ is even and $n$ is odd, then if $m \geq 4$ we have $C_{m,n} = C_{m-2,n} \oplus H_2(-1) \oplus \Gamma_1^{\oplus 2(n-1)}$. Again, a calculation in \textsc{Macaulay 2} shows that generically the fibers of the map $\mb C^{4 \times 2}\to \mb C^{2 \times 2}, A \mapsto A^\top (H_2(-1) \oplus \Gamma_1^{\oplus 2}) A$ are $5$-dimensional. If $n \geq 5$, then $C_{m,n} = A_{m,n-2} \oplus H_4(1) \oplus \Gamma_1^{\oplus 2(m-2)}$ and we conclude as in the case $m,n$ even.\par
    Finally, note that for matrices $A, B, X, Y$ we have $(A \oplus B)^\top (X \oplus Y) (A \oplus B) = A^\top X A \oplus B^\top Y B$. Thus by the induction hypothesis we conclude in all cases that $\mc M_{d+2,m,n} = \mb C^{d + 2, d+ 2}$.
\end{proof}

For $m=1$ or $n=1$ we obtain a rank one symmetric part according to the path framework,  \cite[eq. (4)]{bib:AFS2019}. \par

\begin{figure}[h]
    \centering
    \begin{subfigure}{0.27\textwidth}
        \centering
        $\input{dims_matrices_d/dims_d4}$
        \caption*{$d=4$}
    \end{subfigure}
    \begin{subfigure}{0.31\textwidth}
        \centering
        $\input{dims_matrices_d/dims_d5}
        $
        \caption*{$d=5$}
    \end{subfigure}
    \begin{subfigure}{0.37\textwidth}
        \centering
        $\input{dims_matrices_d/dims_d6}
        $
        \caption*{$d=6$}
    \end{subfigure}
\vspace{0.2cm}
            \centering
            \vspace{0.2cm}
    \begin{subfigure}{0.45\textwidth}
        \centering
        $\input{dims_matrices_d/dims_d7}
        $
        \caption*{$d=7$}
    \end{subfigure}
    \begin{subfigure}{0.48\textwidth}
        \centering
        $\input{dims_matrices_d/dims_d8}
        $
        \caption*{$d=8$}
    \end{subfigure}

    \caption{The dimensions $\dim(\mathcal{M}_{d,m,n})$ stored in $d\times d$ matrices for varying $d$, where bold numbers are \emph{not} covered by \Cref{thm:mainResult} or \Cref{thm:pw-lin matr dim}. We omit the values once the full dimension has been reached. We compute these remaining dimensions according to \Cref{rem:explicitCompForMissingDims}.}
    \label{fig:dim_matrices}.
\end{figure}

\begin{remark}\label{rem:explicitCompForMissingDims}
    Some choices of $m,n,d$ are neither covered by \Cref{thm:mainResult} nor by \Cref{thm:pw-lin matr dim}. Using \textsc{Macaulay 2} with \texttt{NumericalImplicitization} \cite{chen2019numerical}, we computed the missing dimensions of $\mc M_{d,m,n}$ for $m,n \leq d$. We summarize our findings in \Cref{fig:dim_matrices}.
\end{remark}

\begin{example}\label{ex:bensPolyWith100Terms}
    Due to \Cref{fig:dim_matrices}, $\mc M_{5,3,2}$ is a hypersurface in $\mb R^{5\times 5}$. Using the packages \texttt{NumericalImplicitization} and \texttt{MultigradedImplicitization} \cite{cummings2025computing}, one observes that this hypersurface is cut out by a single quintic $f$ with $999$ terms. One checks that $f \in \mathsf{minors}(3,X) \cap \mathsf{minors}(4,X^\mathsf{sk}) \cap \mathsf{minors}(1,X^\mathsf{sym})$, where $\mathsf{minors}(i,X)$ denotes the ideal of $i$-minors of $X$ from the congruence orbit.
\end{example}

\section{Outlook and future work}\label{sec:outlook}

Our notion of signature was introduced by \cite{GLNO2022,diehl2024signature} under the name of \emph{$\mathsf{id}$-signature}, see \Cref{rem:sig_versions} for a detailed discussion. 
Both references observe that this $\mathsf{id}$-signature is not universal.
We prove that in general the signature matrix of a membrane does not satisfy \emph{any} algebraic relations, so in particular, no shuffle relations. 
This proof relies on matrix theory and is not generalizable for tensors in an obvious way. We conjecture that our result is true for all signature tensors.

\begin{conjecture}
\Cref{thm:mainResult} holds for arbitrary level $k\geq 3$.
\end{conjecture}
We provide computational evidence for $d = k=3$ in \Cref{fig:level3dim}.

\begin{figure}[h]
    \centering
    \begin{tabular}{c|ccccc}
   $m \!\setminus \!n$ & 1 & 2 & 3 & 4 & 5 \\ \hline
                 1  & 3 & 6 & 9 & 12 & 14 \\
                 2  & 6 & 12 & 18 & 24 & 27 \\
                 3  & 9 & 18 & 27 &  &  \\
                 4  & 12 & 24 &  &  & \\
                 5  & 14 & 27  &  &  & \\
    \end{tabular}
    \caption{Dimensions of the Zariski closure of third level signature tensors of piecewise bilinear membranes of order $(m,n)$ in $\mathbb R^3$, where we omit the values once the full dimension has been reached.}
    \label{fig:level3dim}
\end{figure}

Already in the path framework, a signature matrix is not enough to recover the path, see \cite[Example 6.8]{bib:AFS2019}. In particular, we can not recover a membrane from its signature matrix in general. An important property of the `right' notion of signature is the ability of recovering membranes from \textit{all} of its signature tensors, as it is done in \cite{Chen1954IteratedIA} for paths. It is an open question whether our signature tensors have this property.  
\begin{question}
    Can we always recover membranes from its signature tensors? 
\end{question}
One can also ask the following more special question:
\begin{question}\label{q:recoveryProblem}
    Can piecewise bilinear membranes of low order already be recovered from a finite set of signature tensors?
\end{question}
Again, this is true for paths, see \cite{bib:PSS2019}.\par
We also remark that recovery with the two-parameter \textit{sums signature} is always possible. For this compare \cite[Theorem 2.20]{DS22} and its proof.

\subsection*{Acknowledgements}
We thank Joscha Diehl for detailed discussions on  various notions of two-parameter signatures, Sven Gro{\ss} for introducing us to the interpolation literature,  Rosa Prei{\ss} for pointing out \cite[Example 4.2]{diehl2024signature}, and Ben Hollering and Bernd Sturmfels  for their help with \Cref{ex:bensPolyWith100Terms,ex:berndspoly}.

\subsection*{Funding sources}
The authors acknowledge support from DFG CRC/TRR 388 ``Rough Analysis, Stochastic Dynamics and Related Fields'', Project A04.

\printbibliography

@article{diehl2024signature,
title = {On the signature of an image},
journal = {Stochastic Processes and their Applications},
volume = {187},
pages = {104661},
year = {2025},
issn = {0304-4149},
doi = {https://doi.org/10.1016/j.spa.2025.104661},
url = {https://www.sciencedirect.com/science/article/pii/S0304414925001024},
author = {Joscha Diehl and Kurusch Ebrahimi-Fard and Fabian N. Harang and Samy Tindel}
}

@article{bib:PSS2019,
  title={Learning Paths from Signature Tensors},
  author={Pfeffer, M. and Seigal, A. and Sturmfels, B.},
  journal={SIAM Journal on Matrix Analysis and Applications},
  year={2019}
}

@article{bib:AFS2019,
  title={Varieties of signature tensors},
  author={Améndola, C. and Friz, P. and Sturmfels, B.},
  journal={Forum of Mathematics, Sigma},
  volume={7},
  year={2019},
  publisher={Cambridge University Press}
}

@article{GLNO2022,
  doi = {10.48550/ARXIV.2202.00491},
  
  url = {https://arxiv.org/abs/2202.00491},
  
  author = {Giusti, Chad and Lee, Darrick and Nanda, Vidit and Oberhauser, Harald},
  title = {A Topological Approach to Mapping Space Signatures},
journal = {Advances in Applied Mathematics (to appear). }, 
  archiveprefix = {arXiv},
    eprint = {2202.00491},
    primaryclass = {math.FA},
  year = {2022},
}

@article{ZLT22,
  title={Two-dimensional signature of images and texture classification},
  author={Zhang, Sheng and Lin, Guang and Tindel, Samy},
  journal={Proceedings of the Royal Society A},
  volume={478},
  number={2266},
  pages={20220346},
  year={2022},
  publisher={The Royal Society},
  doi={10.1098/rspa.2022.0346},
}

@misc{DS22,
      title={Two-parameter sums signatures and corresponding quasisymmetric functions}, 
      author={Joscha Diehl and Leonard Schmitz},
      year={2022},
      eprint={2210.14247},
      archivePrefix={arXiv},
      primaryClass={math.CO}
}

@misc{CDEFT2024,
      title={A multiplicative surface signature through its {Magnus} expansion}, 
      author={Ilya Chevyrev and Joscha Diehl and Kurusch Ebrahimi-Fard and Nikolas Tapia},
      year={2024},
      eprint={2406.16856},
      archivePrefix={arXiv},
      primaryClass={math.RA},
      url={https://arxiv.org/abs/2406.16856}, 
}

@misc{lee2024,
      title={The Surface Signature and Rough Surfaces}, 
      author={Darrick Lee},
      year={2024},
      eprint={2406.16857},
      archivePrefix={arXiv},
      primaryClass={math.FA},
      url={https://arxiv.org/abs/2406.16857}, 
}

@misc{LO2023,
      title={Random Surfaces and Higher Algebra}, 
      author={Darrick Lee and Harald Oberhauser},
      year={2023},
      eprint={2311.08366},
      archivePrefix={arXiv},
      primaryClass={math.PR},
      url={https://arxiv.org/abs/2311.08366}, 
}

@article{CLOS07,
  author = {Cox, David and Little, John and O’Shea, Donal},
  description = {Ideals, Varieties, and Algorithms | SpringerLink},
  title = {{Ideals, Varieties, and Algorithms. An Introduction to Computational Algebraic Geometry and Commutative Algebra}},
  url = {https://link.springer.com/book/10.1007/978-0-387-35651-8},
  year = 2007
}

@misc{galuppi2024ranksymmetriessignaturetensors,
      title={Rank and symmetries of signature tensors}, 
      author={Francesco Galuppi and Pierpaola Santarsiero},
      year={2024},
      eprint={2407.20405},
      archivePrefix={arXiv},
      primaryClass={math.AG},
      url={https://arxiv.org/abs/2407.20405}, 
}

@article{FLS24,
author = {Friz, Peter K. and Lyons, Terry and Seigal, Anna},
title = {Rectifiable paths with polynomial log-signature are straight lines},
journal = {Bulletin of the London Mathematical Society},
pages = {},
year = {2024},
doi = {https://doi.org/10.1112/blms.13110},
url = {https://londmathsoc.onlinelibrary.wiley.com/doi/abs/10.1112/blms.13110},
eprint = {https://londmathsoc.onlinelibrary.wiley.com/doi/pdf/10.1112/blms.13110}
}

@article{TERAN201144,
title = {The solution of the equation {$XA+AX^\top=0$} and its application to the theory of orbits},
journal = {Linear Algebra and its Applications},
volume = {434},
number = {1},
pages = {44-67},
year = {2011},
issn = {0024-3795},
doi = {https://doi.org/10.1016/j.laa.2010.08.005},
url = {https://www.sciencedirect.com/science/article/pii/S0024379510004131},
author = {Fernando {De Terán} and Froilán M. Dopico}
}

@article{HORN20061010,
title = {Canonical forms for complex matrix congruence and *congruence},
journal = {Linear Algebra and its Applications},
volume = {416},
number = {2},
pages = {1010-1032},
year = {2006},
issn = {0024-3795},
doi = {https://doi.org/10.1016/j.laa.2006.01.005},
url = {https://www.sciencedirect.com/science/article/pii/S0024379506000383},
author = {Roger A. Horn and Vladimir V. Sergeichuk}
}

@article{BRUALDI198531,
title = {Combinatorial verification of the elementary divisors of tensor products},
journal = {Linear Algebra and its Applications},
volume = {71},
pages = {31-47},
year = {1985},
issn = {0024-3795},
doi = {https://doi.org/10.1016/0024-3795(85)90233-2},
url = {https://www.sciencedirect.com/science/article/pii/0024379585902332},
author = {Richard A. Brualdi}
}

@article{Chen1954IteratedIA,
  title={{Iterated Integrals and Exponential Homomorphisms}},
  author={Kuo-Tsai Chen},
  journal={Proceedings of The London Mathematical Society},
  year={1954},
  pages={502-512},
  url={https://api.semanticscholar.org/CorpusID:121288919}
}

@misc{Gyu13,
    author = {Gyurkó, Lajos Gegerly and Lyons, Terry and Kontkowski, Mark and Field, Jonathan} ,
    title = {Extracting information from the signature of a financial data stream},
    archiveprefix = {arXiv},
    year = {2013},
    eprint = {1307.7244}, 
    primaryclass = {q-fin.ST},
    doi = {10.48550/ARXIV.1307.7244},
    url = {https://arxiv.org/abs/1307.7244},
}

@inproceedings{handw13,
    author = {Jin, Lianwen and Lyons, Terry and Ni, Hao and Yang, Weixin},
    title = {Rotation-free online handwritten character
recognition using dyadic path signature features, hanging normalization, and deep neural networks},
    booktitle = {23rd
International Conference on Pattern Recognition (ICPR)},
    year = {2016},
    publisher = {IEEE}
}

@misc{preiß2024algebraicgeometrypathsiteratedintegrals,
      title={An algebraic geometry of paths via the iterated-integrals signature}, 
      author={Rosa Preiß},
      year={2024},
      eprint={2311.17886},
      archivePrefix={arXiv},
      primaryClass={math.RA},
      url={https://arxiv.org/abs/2311.17886}, 
}

@article{clausel2024barycenterfreenilpotentlie,
author = {Clausel, Marianne and Diehl, Joscha and Mignot, Raphael and Schmitz, Leonard and Sugiura, Nozomi and Usevich, Konstantin},
title = {The Barycenter in Free Nilpotent Lie Groups and Its Application to Iterated-Integrals Signatures},
journal = {SIAM Journal on Applied Algebra and Geometry},
volume = {8},
number = {3},
pages = {519-552},
year = {2024},
doi = {10.1137/23M159024X},

URL = { 
    
        https://doi.org/10.1137/23M159024X
    
    

},
eprint = { 
    
        https://doi.org/10.1137/23M159024X
    
    

}
}

@article{AGOSS2025,
title = {Decomposing tensor spaces via path signatures},
journal = {Journal of Pure and Applied Algebra},
volume = {229},
number = {1},
pages = {107807},
year = {2025},
issn = {0022-4049},
doi = {https://doi.org/10.1016/j.jpaa.2024.107807},
url = {https://www.sciencedirect.com/science/article/pii/S0022404924002044},
author = {Carlos Améndola and Francesco Galuppi and Ángel David {Ríos Ortiz} and Pierpaola Santarsiero and Tim Seynnaeve}
}

@unpublished{mignot:hal-04392568,
  TITLE = {{Principal geodesic analysis for time series encoded with signature features}},
  AUTHOR = {Mignot, Raphael and Clausel, Marianne and Usevich, Konstantin},
  URL = {https://hal.science/hal-04392568},
  HAL_LOCAL_REFERENCE = {BioSiS},
  YEAR = {2024},
  MONTH = Jan,
  HAL_ID = {hal-04392568},
  HAL_VERSION = {v1},
}

@InProceedings{10.1007/978-3-031-38271-0_45,
author="Am{\'e}ndola, Carlos
and Lee, Darrick
and Meroni, Chiara",
editor="Nielsen, Frank
and Barbaresco, Fr{\'e}d{\'e}ric",
title="Convex Hulls of Curves: Volumes and Signatures",
booktitle="Geometric Science of Information",
year="2023",
publisher="Springer Nature Switzerland",
address="Cham",
pages="455--464",
isbn="978-3-031-38271-0"
}

@article{DEFT22,
author = {Diehl, Joscha and Ebrahimi-Fard, Kurusch and Tapia, Nikolas},
title = {Tropical Time Series, Iterated-Sums Signatures, and Quasisymmetric Functions},
journal = {SIAM Journal on Applied Algebra and Geometry},
volume = {6},
number = {4},
pages = {563-599},
year = {2022},
doi = {10.1137/20M1380041},

URL = { 
    
        https://doi.org/10.1137/20M1380041
    
    

},
eprint = { 
    
        https://doi.org/10.1137/20M1380041
    
    

}
}

@article{diehl2023fruitsfeatureextractionusing,
  title={FRUITS: feature extraction using iterated sums for time series classification},
  author={Diehl, Joscha and Krieg, Richard},
  journal={Data Mining and Knowledge Discovery},
  volume={38},
  number={6},
  pages={4122--4156},
  year={2024},
  publisher={Springer}
}

@article{Diehl2019,
  author    = {Joscha Diehl and Jeremy Reizenstein},
  title     = {Invariants of Multidimensional Time Series Based on Their Iterated-Integral Signature},
  journal   = {Acta Applicandae Mathematicae},
  year      = {2019},
  volume    = {164},
  number    = {1},
  pages     = {83--122},
  doi       = {10.1007/s10440-018-00227-z},
  url       = {https://doi.org/10.1007/s10440-018-00227-z},
  issn      = {1572-9036},
}

@book{zienkiewicz_finite_2005,
  author = {Zienkiewicz, O. C. and Taylor, R. L. and Zhu, J. Z.},
  edition = 6,
  isbn = {0750663200},
  month = may,
  publisher = {Butterworth-Heinemann},
  shorttitle = {The Finite Element Method},
  title = {The Finite Element Method: Its Basis and Fundamentals, Sixth Edition},
  year = 2005
}

@book{elman2014finite,
  title={Finite elements and fast iterative solvers: with applications in incompressible fluid dynamics},
  author={Elman, Howard and Silvester, David and Wathen, Andy},
  year={2014},
  publisher={OUP Oxford}
}

@book{bovik2009essential,
  title={The essential guide to image processing},
  author={Bovik, Alan C},
  year={2009},
  publisher={Academic Press}
}

@misc{GraysonStillmanMacaulay2,
  author       = {David Grayson and Michael Stillman},
  title        = {\textsc{Macaulay 2}, a software system for research in algebraic geometry},
  howpublished = {\url{https://macaulay2.com}},
  note         = {Available online},
}

@book{eisenbud2001computations,
  title={Computations in algebraic geometry with {\textsc{Macaulay 2}}},
  author={Eisenbud, David and Grayson, Daniel R and Stillman, Mike and Sturmfels, Bernd},
  volume={8},
  year={2001},
  publisher={Springer Science \& Business Media}
}

@article{GALUPPI2019282,
title = {{The Rough Veronese variety}},
journal = {Linear Algebra and its Applications},
volume = {583},
pages = {282-299},
year = {2019},
issn = {0024-3795},
doi = {https://doi.org/10.1016/j.laa.2019.08.029},
url = {https://www.sciencedirect.com/science/article/pii/S0024379519303775},
author = {Francesco Galuppi}
}

@book{bochnak2013real,
  title={Real algebraic geometry},
  author={Bochnak, Jacek and Coste, Michel and Roy, Marie-Fran{\c{c}}oise},
  volume={36},
  year={2013},
  publisher={Springer Science \& Business Media}
}

@article{HARRIS198471,
title = {On symmetric and skew-symmetric determinantal varieties},
journal = {Topology},
volume = {23},
number = {1},
pages = {71-84},
year = {1984},
issn = {0040-9383},
doi = {https://doi.org/10.1016/0040-9383(84)90026-0},
url = {https://www.sciencedirect.com/science/article/pii/0040938384900260},
author = {Joe Harris and Loring W. Tu}
}

@article{chen2019numerical,
  title={Numerical implicitization},
  author={Chen, Justin and Kileel, Joe},
  journal={Journal of Software for Algebra and Geometry},
  volume={9},
  number={1},
  pages={55--63},
  year={2019},
  publisher={Mathematical Sciences Publishers}
}

@article{cummings2025computing,
  title={Computing implicitizations of multi-graded polynomial maps},
  author={Cummings, Joseph and Hollering, Benjamin},
  journal={Journal of Symbolic Computation},
  pages={102459},
  year={2025},
  publisher={Elsevier}
}

\end{document}